\documentclass[fleqn,a4paper,10pt]{amsart}

\usepackage{graphicx}
\usepackage{tikz}
\usetikzlibrary{shapes,decorations.pathmorphing}
\usepackage{amsmath}
\usepackage{amssymb,amsthm}
\usepackage{comment}
\usepackage{hyperref}
\usepackage{amsaddr}

\def\cB{\mathcal{B}}
\def\cH{\mathcal{H}}
\def\cM{\mathcal{M}}
\def\mH{\cM\cH}
\def\mini{\mathrm{min}}
\def\maxi{\mathrm{max}}

\def\uctau{c^{\uparrow}(\tau)}
\def\lctau{c^{\downarrow}(\tau)}
\def\ubtau{\overline{\tau}^{\uparrow}}
\def\lbtau{\overline{\tau}^{\downarrow}}

\def\Lu{\L{}ukasiewicz\ } 

\newtheorem{claim}{Claim}
\newtheorem{theo}{Theorem}
\newtheorem{prop}{Proposition}
\newtheorem{coro}{Corollary}
\theoremstyle{definition}
\newtheorem{rmk}{Remark}

\title{On the two-point function of general planar maps and hypermaps}
\author{J\'er\'emie Bouttier$^{1,2}$, \'Eric Fusy$^3$ and Emmanuel Guitter$^1$}
\address{$^1$Institut de Physique Th\'eorique\\ CEA, IPhT, 91191 Gif-sur-Yvette, France\\ CNRS, URA 2306\\
  $^2$D\'epartement de Math\'ematiques et Applications\\ \'Ecole normale sup\'erieure,
  45 rue d'Ulm, 75231 Paris Cedex 05, France\\
  $^3$ LIX, \'Ecole Polytechnique, 91128 Palaiseau, France}
\email{jeremie.bouttier@cea.fr, fusy@lix.polytechnique.fr, emmanuel.guitter@cea.fr}

\date{\today}

\begin{document}
\maketitle

\begin{abstract}
  We consider the problem of computing the distance-dependent two-point function of general planar maps and hypermaps, i.e.~the problem of counting such maps with two marked points at a prescribed distance. The maps considered here may have faces of arbitrarily large degree, which requires new bijections to be tackled. We obtain exact expressions for the following cases: general and bipartite maps counted by their number of edges, $3$-hypermaps and $3$-constellations counted by their number of dark faces, and finally general and bipartite maps counted by both their number of edges and their number of faces.
\end{abstract}

\section{Introduction}

Much attention has been devoted to the study of metric properties of random planar maps,
starting from the physical predictions of Ambj\o rn and Watabiki \cite{AmWa95} and the seminal probabilist work of Chassaing and Schaeffer \cite{ChSc04}.
Of particular interest is the so-called two-point function which, colloquially speaking, encodes the distribution of the distance between two random points. This two-point function has been computed exactly \cite{GEOD,PDFRaman} for several families of maps: quadrangulations, triangulations, Eulerian triangulations... A general formalism was developed in \cite{BG12} to address the case of maps with controlled face degrees but, in practice, a fully explicit expression was only obtained in the case of bounded face degrees. In this paper, we consider instead families of maps with \emph{unbounded} face degrees: general (arbitrary) planar maps, bipartite planar maps and, more generally, hypermaps and constellations. The control is then on their natural size parameter, namely their number of edges. We will compute exactly the two-point function of these maps, using as a new ingredient some bijections which relate them to maps with bounded degrees, and keep track of distances (in the case of hypermaps, we actually consider a ``quasi-distance'' based on oriented edges). More precisely, we shall extend the distance-preserving bijection recently found by Ambj\o rn and Budd \cite{AmBudd} between planar quadrangulations and general planar maps, and get a correspondence between, on the one hand, bipartite maps with controlled face degrees and, on the other hand, hypermaps with controlled hyperedge degrees but arbitrary (uncontrolled) face degrees. Briefly said, the Ambj\o rn-Budd bijection consists in applying on quadrangulations the rules opposite to those used for the Schaeffer bijection \cite{ChSc04,SchPhD}
and here we shall apply the same trick to the more general rules used for the BDG bijection \cite{BDG04}.

\begin{figure}
  \centering
  \begin{tikzpicture}[
    set/.style = {draw, very thick, rectangle,
      rounded corners, inner sep=10pt, inner ysep=10pt, align=center},
    mset/.style = {set, font=\bfseries},
    bij/.style = {<->, very thick},
    spec/.style = {->, decorate,
      decoration={snake,amplitude=.4mm,segment length=2mm,post
        length=1mm}, align=left}]
    \useasboundingbox (-6,4) rectangle (6,-13);

    \draw (0,0) node[mset](slm){Suitably labelled\\ maps ($\cB$)};
    \draw (-5.5,0) node[mset](wlhm){Well-labelled\\ hypermaps
      ($\cH$)};
    \draw (5.5,0) node[mset](mwlhm){Mirror\\ well-labelled\\
      hypermaps ($\mH$)};
    \draw (slm) edge[bij] node[above]{\Large $\Phi$}
      node[below]{Thm.~\ref{theo:bijmobile}}(wlhm);
    \draw (slm) edge[bij] node[above]{\Large $\Phi^-$} 
      node[below]{Cor.~\ref{coro:bij_anti_mobile}}(mwlhm);

    \draw (0,3) node[set,dashed](slq){Suitably labelled \\
      quadrangulations};
    \draw (-5.5,3) node[set,dashed](wlm){Well-labelled maps};
    \draw (5.5,3) node[set,dashed](wlm2){Well-labelled maps};
    \draw (slq) edge[bij] node[above]{(M,AB)} node[below]{
      Rmk.~\ref{rmk:mab}}(wlm);
    \draw (slq) edge[bij] node[above]{(M,AB)} node[below]{
      Rmk.~\ref{rmk:mab}}(wlm2);
    \draw (slm) edge[spec] (slq);
    \draw (wlhm) edge[spec] (wlm);
    \draw (mwlhm) edge[spec] (wlm2);

    \draw (0,-3) node[set](vpbm){Vertex-pointed \\ bipartite maps};
    \draw (-5.5,-3) node[set](mob){Mobiles};
    \draw (5.5,-3) node[set](vphm){Vertex-pointed \\ hypermaps};
    \draw (vpbm) edge [bij] node[above]{
      Prop.~\ref{prop:point_bip_mobiles} (BDG)} (mob);
    \draw (vpbm) edge[bij] node[above]{Prop.~\ref{prop:bij_bip_hyp}}(vphm);
    \draw (mob) edge[bij,bend right=20] node[near start,below]{
      Thm.~\ref{theo:hyp_mobiles}} (vphm);
    \draw (slm) edge[spec] node[right]{geodesic\\ labelling} (vpbm);
    \draw (wlhm) edge[spec] node[right]{single face} (mob);
    \draw (mwlhm) edge[spec] node[right]{geodesic\\ labelling} (vphm);

    \draw (0,-6) node[set](vp2p){Vertex-pointed\\
      $2p$-angulations};
    \draw (-5.5,-6) node[set](pmob){$p$-mobiles};
    \draw (5.5,-6) node[set] (vpphm) {Vertex-pointed\\
      $p$-hypermaps}; 
    \draw (vp2p) edge [bij] node[above]{(BDG)} node[below]{
      (S for $p=2$)} (pmob);
    \draw (vp2p) edge[bij] node[below]{(AB for $p=2$)} (vpphm);
    \draw (pmob) edge[bij,bend right=20] node[near start,below]{
      Thm.~\ref{theo:hyp_mobiles} (AB for $p=2$)}
    node[near end,below,text=blue]{
      Secs.~\ref{sec:2pgen}, \ref{sec:2p3h}} (vpphm);
    \draw (vpbm) edge[spec] (vp2p); \draw (mob) edge[spec] (pmob);
    \draw (vphm) edge[spec] (vpphm);

    \draw (0,-9) node[set](vpc2p){Vertex-pointed\\ stretched\\
      $2p$-angulations};
    \draw (-5.5,-9) node[set](pdmob){$p$-descending \\ mobiles};
    \draw (5.5,-9) node[set](vppc){Vertex-pointed\\ $p$-constellations};
    \draw (vpc2p) edge [bij] (pdmob); \draw (vpc2p) edge[bij] (vppc);
    \draw (pdmob) edge[bij,bend right=20] node[near start,below]{
      Prop.~\ref{prop:const1}} node[near end,below,text=blue]{
      Secs.~\ref{sec:2pbip}, \ref{sec:2p3c}} (vppc);
    \draw (vp2p) edge[spec] (vpc2p); \draw (pmob) edge[spec] (pdmob);
    \draw (vpphm) edge[spec] (vppc);
    \draw (0,-12) node[set](preg){$(p+1)$-regular \\
      constellations};
    \draw (pdmob) edge[bij,out=270,in=180] node[below]{
      (BDG)} (preg);
    \draw (vpc2p) edge[bij] (preg);
    \draw (vppc) edge[bij,out=270,in=0] node[below,align=left]{
      Prop.~\ref{prop:const2}\\ \textcolor{blue}{
        Sec.~\ref{sec:comm2pconst}}} (preg);
  \end{tikzpicture}
  \caption{An overview of the bijections presented in
    Section~\ref{sec:bij}. Bijections are represented by lines with
    double arrows and wiggled lines represent specializations.
    Where applicable, we refer to the text
    (with blue text corresponding to applications) or to previous works:
    S stands for the Schaeffer bijection \cite{SchPhD}, BDG for the bijection
    of \cite{BDG04}, M for the Miermont bijection \cite{Miermont2009} and AB 
    for bijections introduced in \cite{AmBudd}.}
  \label{fig:overview}
\end{figure}
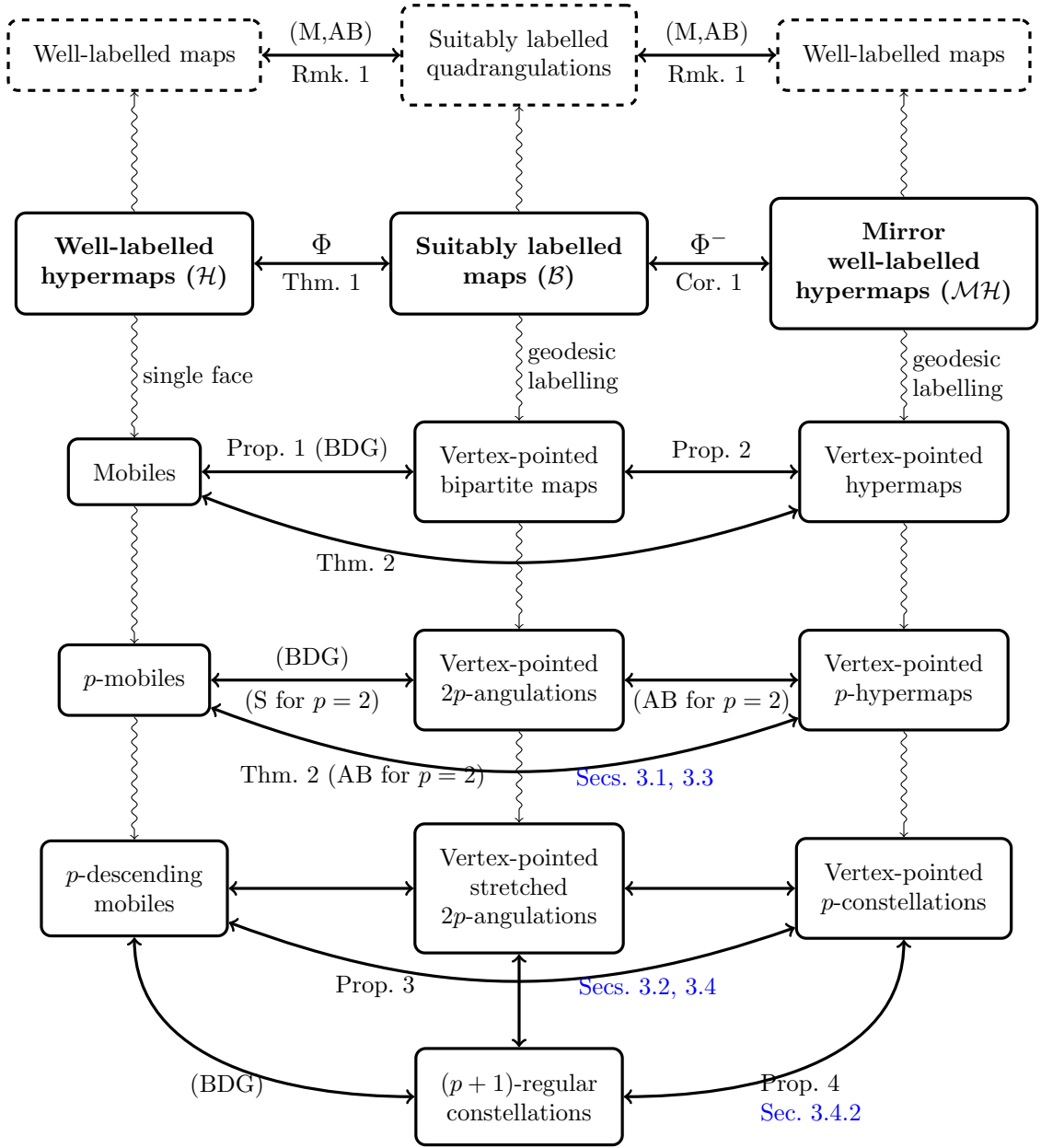

The paper is organized as follows. All the necessary bijections are established in Section~\ref{sec:bij} (Figure~\ref{fig:overview} provides an overview intended to help the reader). We start by introducing so-called suitably labelled maps and well-labelled hypermaps (Section~\ref{sec:slm}) and a very general bijection $\Phi$ between them, using the BDG rules (Section~\ref{sec:slmbij}). Reversing these rules gives rise to another ``mirror'' bijection $\Phi^-$ (Section~\ref{sec:mirror}). We will then specialize $\Phi$ and $\Phi^-$ to vertex-pointed bipartite maps endowed with their geodesic labelling (Section~\ref{sec:mobspec}): $\Phi$ yields the BDG bijection, $\Phi^-$ is the new ingredient mentioned above, and their composition yields a direct distance-preserving correspondence between hypermaps and mobiles. We then discuss the further specialization to constellations (Section~\ref{sec:conspec}), before completing some technical proofs (Section~\ref{sec:proof_bij}) and discussing the extension to maps of higher genus (Section~\ref{sec:highergenus}). In Section~\ref{sec:2psing}, we apply these bijections to compute explicitly the two-point function of several families of maps controlled by their number of edges: general maps (Section~\ref{sec:2pgen}), bipartite maps and general hypermaps (Section~\ref{sec:2pbip}), $3$-hypermaps (Section~\ref{sec:2p3h}) and $3$-constellations (Section~\ref{sec:2p3c}). Connections to previous works \cite{GEOD,PDFRaman,BG12} are discussed where applicable. In Section~\ref{sec:2p2par}, we provide a refinement of the two-point function of general maps (Section~\ref{sec:2pgen2par}) and bipartite maps (Section~\ref{sec:2pbip2par}), where we control both their number of edges and their number of faces. We conclude in Section~\ref{sec:conclusion} by some remarks and open questions.

\section{Bijections}
\label{sec:bij}

\subsection{Suitably labelled maps and well-labelled hypermaps: definitions}
\label{sec:slm}
A \emph{suitably labelled map} is a map $B$ (on the sphere) where each vertex $v$ carries a label $\ell(v)\in\mathbb{Z}$,  
such that, for any edge $e=\{u,v\}$ of $B$, $|\ell(v)-\ell(u)|=1$. 
Note that $B$ is necessarily bipartite (each edge connects a vertex of odd label to a vertex of even label). 
A \emph{local max} (resp. \emph{local min}) is a vertex such that all neighbors have smaller (resp. greater) label. The \emph{cw-type} (resp. \emph{ccw-type}) 
of a face $f$ is the cyclic list of integers given by the labels
of vertices in clockwise (resp. counterclockwise) order around $f$. 
Denote by $\cB$ the set of suitably labelled maps.

An Eulerian map is a map (on the sphere) with all vertices of even degree; such maps can be properly 
bicolored at their faces (with dark faces and light
faces), i.e., such that any edge has a dark face on one side and a light face on the other side. 
An \emph{hypermap} $H$ is a properly face-bicolored Eulerian map (viewing dark faces as hyperedges).
The \emph{star-representation} of $H$
is the bipartite map $S$ obtained by replacing the contour of each dark face $f$ (of a given degree $d$) 
by a star (of degree $d$) centered at a new black vertex $v_f$ placed inside $f$, 
see Figure~\ref{fig:hypermap}. 
\begin{figure}[h!]
\begin{center}
\includegraphics[width=10cm]{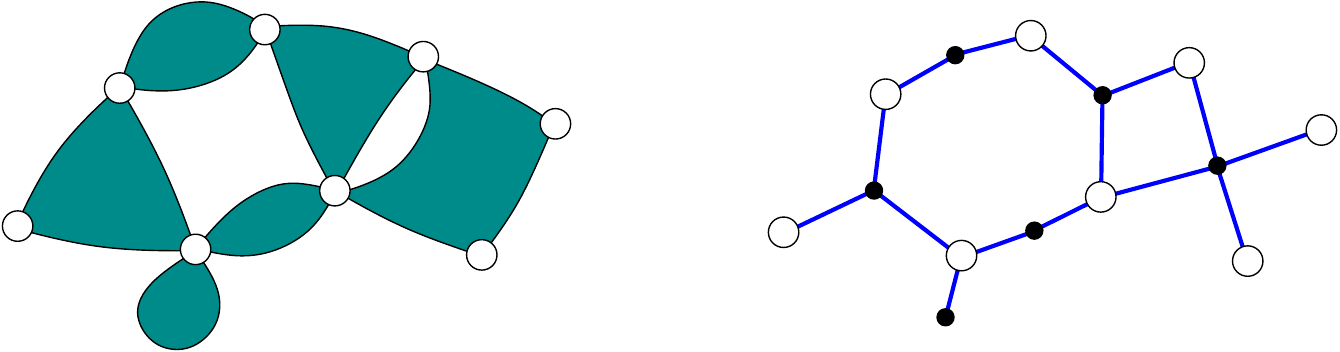}
\end{center}
\caption{Left: an hypermap, right: its star-respresentation.}
\label{fig:hypermap}
\end{figure}
A \emph{well-labelled hypermap} is an hypermap $H$ (on the sphere) where each vertex $v$ carries a label $\ell(v)\in\mathbb{Z}$ 
such that, for any edge $e=\{v,u\}$, with the dark face incident to $e$ on the right of $e$ traversed from $v$ to $u$, $\ell(u)\geq \ell(v)-1$. 
For a vertex $u$ of $H$, a \emph{right neighbor} of $u$ is a vertex $v$ adjacent to $u$, and such that $\{v,u\}$
traversed from $v$ to $u$ has a dark face on its right. A \emph{right local min} 
(resp. \emph{right local max}) 
 of $H$ is a vertex $u$ such that any right neighbor $v$ of
$u$ satisfies $\ell(v)\geq \ell(u)$ (resp. $\ell(v)\leq \ell(u)$).   
The \emph{cw-type} $\tau$ of a dark face $f$ is the cyclic list of integers given by the labels of vertices
in \emph{clockwise} order around $f$.   
Denote by $\cH$ the family of well-labelled hypermaps.  
Note that by definition, the cw-type $\tau$ of a dark face is a so-called \emph{\Lu cyclic sequence},
i.e., a cyclic integer list such that the difference between an element of the list and the preceding one is at least $-1$.  
Define the \emph{upper completion} $\uctau$ (resp. \emph{lower completion} $\lctau$)  
of $\tau$ as the cyclic sequence  
obtained from $\tau$ by inserting the rising sequence $i+1,\ldots,j+1$ (resp. the rising
sequence $i-1,\ldots,j-1$) between any two
consecutive elements $i,j$ such that $j\geq i$. The \emph{upper complement} $\ubtau$ (resp. \emph{lower complement} $\lbtau$) of $\tau$
is the cyclic sequence $\uctau\backslash\tau$ (resp. $\lctau\backslash\tau$) 
taken in reverse order, see Figure~\ref{fig:type}
for an example. Note that the upper or lower complement is also a \Lu cyclic sequence, and that 
the mappings from a \Lu cyclic sequence to its upper and lower complements are inverse of one another. 

\begin{figure}[h!]
\begin{center}
\includegraphics[width=11cm]{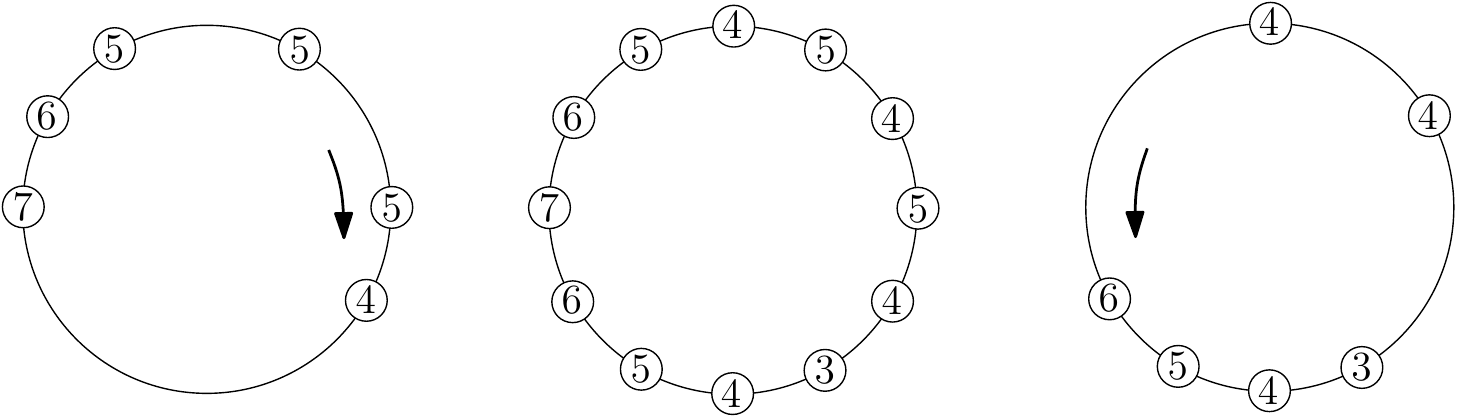}
\end{center}
\caption{Left: a \Lu cyclic sequence $\tau$, right: the lower complement $\lbtau$,
middle: the lower completion of $\tau$ (read clockwise) 
which is also the upper completion  of $\lbtau$ (read counterclockwise).} 
\label{fig:type}
\end{figure}

\subsection{Suitably labelled maps and well-labelled hypermaps: the bijection.}
\label{sec:slmbij}
At first we explain how to construct (the star representation of)  a well-labelled hypermap from a   
 suitably labelled map.  
Let $B\in\cB$. Place a black vertex $v_f$ inside each face $f$ of $B$. Then apply the so-called \emph{BDG rules}
(shown in Figure~\ref{fig:rules} left-part)  
in $f$, i.e., when traversing $f$ clockwise, for each descending edge $e=\{u,v\}$, insert a new edge from $v_f$ to $u$ 
(note that the vertices of $B$ not incident to any of these new edges are exactly the local min of $B$). 
Then erase all the local min of $B$ and all edges of $B$. Call $S$ the obtained figure. 

\begin{figure}[h!]
\begin{center}
\includegraphics[width=11cm]{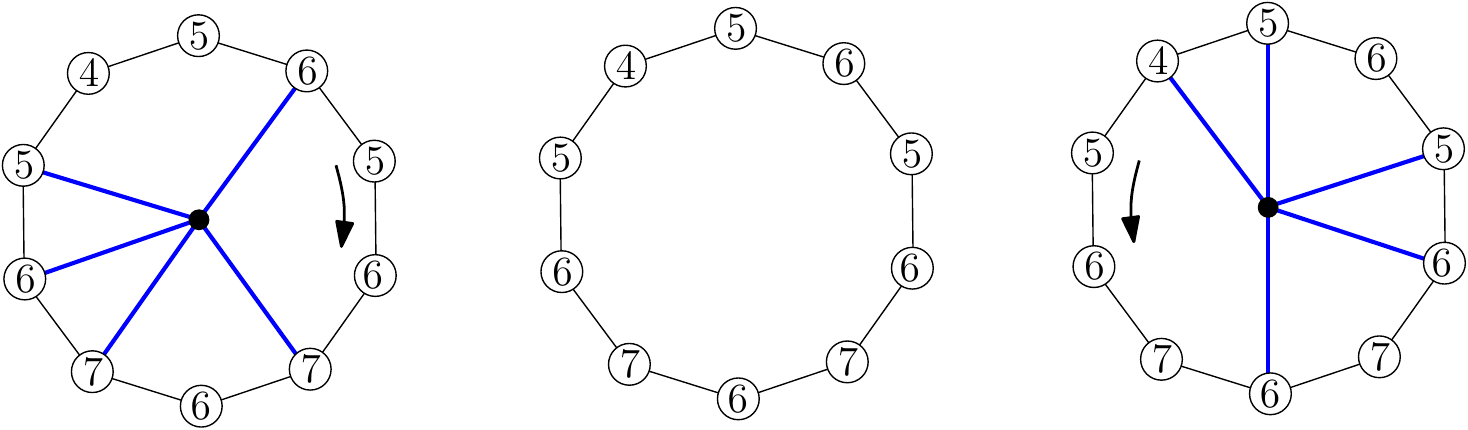}
\end{center}
\caption{Middle: a face $f$ of a suitably labelled map; left: application of the BDG rules,
right: application of the complementary rules. The list of neighbor-labels of the black vertex
in counterclockwise order using the complementary rules is the lower complement of the list of neighbor-labels 
of the black vertex in clockwise order using the BDG rules.  
}
\label{fig:rules}
\end{figure}

\begin{claim}\label{claim:Smap}
The obtained figure $S$ is the star-representation of a well-labelled hypermap $H$. 
\end{claim} 
The proof is delayed to Section~\ref{sec:proof_bij}  
(the local condition of being well-labelled is trivially satisfied, the non-trivial
part is to show that $S$ is a map, i.e., is connected). Let $\Phi$ be the mapping that associates $H$ to $B$. 
Figure \ref{fig:bijopen} displays an example of this mapping.

\begin{figure}[h!]
\begin{center}
\includegraphics[width=13cm]{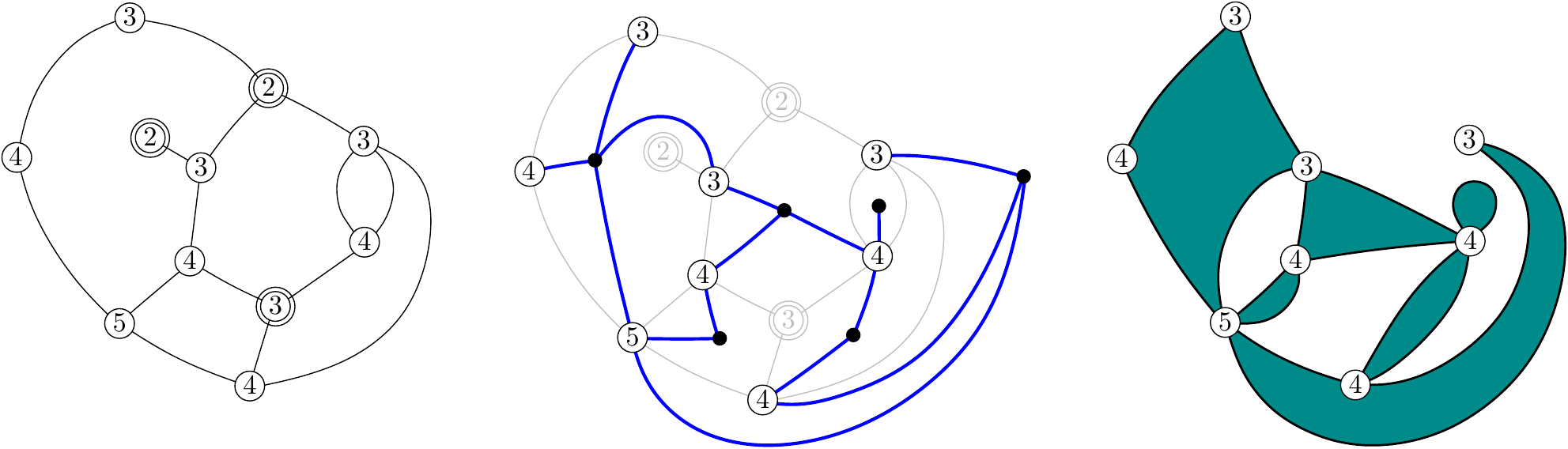}
\end{center}
\caption{The mapping $\Phi$ from a suitably labelled map (local min are surrounded) to a well-labelled hypermap.}
\label{fig:bijopen}
\end{figure}

We now describe the inverse mapping $\Psi$ (see Figure \ref{fig:bijclose} for an example). Let $H\in\cH$.   
For each light face $f$ of $H$, denote by $\mini(f)$ (resp. $\maxi(f)$) the minimal (resp. maximal) label of  vertices incident to $f$. Let $S$ be the star representation of $H$. Consider any face $f$ of $S$ (which identifies 
to a light face of $H$).  Insert inside $f$ a vertex $v_f$ of label $\mini(f)-1$. 
Then, for each corner $c$ of $f$ at a labelled (white) vertex $v$, insert a leg in $c$.
If $\ell(v)>\mini(f)$, connect the free extremity of the leg to the 
next corner of label $\ell(v)-1$ after $c$ in counterclockwise order around $f$ (note that, when the map is drawn in the plane, 
going counterclockwise around the outer face amounts to going clockwise around the map). 
If $\ell(v)=\mini(f)$, connect the free extremity of the leg to $v_f$. Finally delete all black vertices and all edges of $S$.
Denote by $B$ the obtained figure.

\begin{figure}[h!]
\begin{center}
\includegraphics[width=13cm]{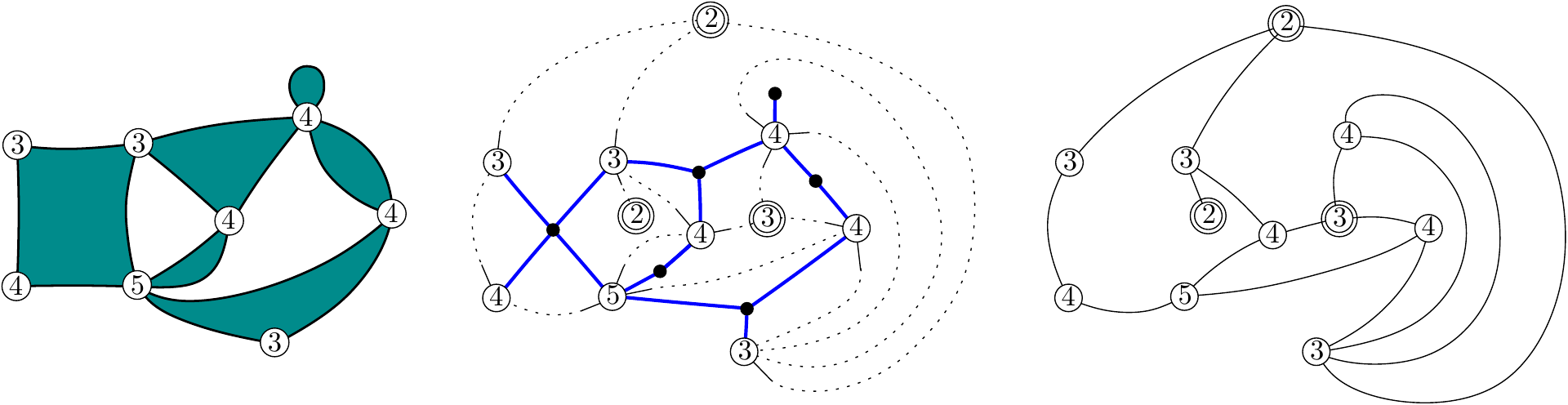}
\end{center}
\caption{The mapping $\Psi$ from a well-labelled hypermap to a suitably labelled map.}
\label{fig:bijclose}
\end{figure}

\begin{claim}\label{claim:Bmap}
The obtained figure $B$ is a suitably labelled map. 
\end{claim} 
The proof is delayed to Section~\ref{sec:proof_bij} (the fact that the labelling is suitable is clear by construction,
the nontrivial point is to show that $B$ is a map). 

\begin{theo}\label{theo:bijmobile}
The mapping $\Phi$ is a bijection  between $\cB$ and $\cH$; the inverse mapping is $\Psi$. 
For $B\in\cB$ and $H=\Phi(B)$, each vertex $v$ of $H$ corresponds to a non local min vertex $v'$ of $B$ of the same label, and $v$ is a right local max in $H$ iff $v'$ is a local max in $B$, 
each light face $f$ of $H$ corresponds to a local min vertex of $B$ of label $\mini(f)-1$, and each dark face of $H$ of cw-type $\tau$
corresponds to a face of $B$ of cw-type $\lctau$.   
\end{theo}
The proof that $\Phi$ and $\Psi$ are inverse of each other is delayed to Section~\ref{sec:proof_bij}. 
The parameter-correspondence follows rather directly from the way the constructions $\Phi$ and $\Psi$ 
are defined. More precisely, the fact that each face $f$ of $H$ corresponds to a local min of $B$ of label $\mini(f)-1$
follows from the definition of $\Psi$, and the fact that each dark face of $H$ corresponds to a face of $B$ of cw-type 
$\lctau$ follows from the description of $\Phi$ (see Figures~\ref{fig:type} and~\ref{fig:rules}). 
Finally, if $v'\in B$ is neither a local min
 nor a local max, then the corresponding $v\in H$ is not a right local max (see Figure~\ref{fig:local_max}(a)), whereas if $v$ is a local
max of $B$, then any right neighbor $u$ of $v$ satisfies $\ell(u)\leq \ell(v)$ (see Figure~\ref{fig:local_max}(b)), 
so that $v$ is a right local max in $H$. 

\begin{figure}
\begin{center}
\includegraphics[width=10cm]{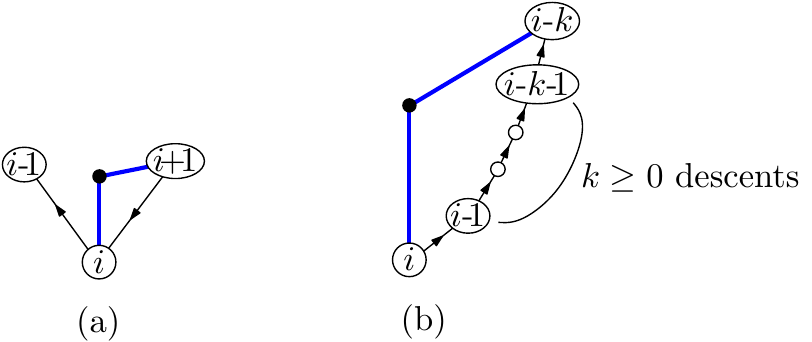}
\end{center}
\caption{(a) A vertex $v$ of $B$ that is neither a local min nor local max is incident to at least one face
of $B$ of the type displayed here, hence has a right neighbor in $H$
of larger label, so that $v$ is not a right local max in $H$. (b) If a vertex $v$ of $B$ is a local max, 
then any right neighbor of $v$ in $H$ has smaller (or equal) label, hence $v$ is a right local max
in $H$.}
\label{fig:local_max}
\end{figure}

\begin{rmk} \label{rmk:mab}
  By restricting to the elements of $\cB$ whose all faces have degree
  $4$, we recover the bijection of \cite[Theorem 1]{AmBudd}
  between suitably labelled quadrangulations and well-labelled maps
  (identified to hypermaps by blowing each of their edges into a dark
  face of degree $2$). Actually, this bijection may also be viewed as
  equivalent to the Miermont bijection in the planar case: indeed, for
  a fixed quadrangulation $Q$, the data of a suitable labelling is
  equivalent to the data of a set of ``sources'' and of a ``delay
  vector'' in the terminology of \cite{Miermont2009}. More precisely, taking the
  local min of the suitable labelling as sources, and their labels as
  delays, it is not difficult to check that the label of any other
  vertex is given by the Miermont prescription, namely
  \begin{equation}
    \ell(v) = \min_{u \text{ local min}} ( \mathrm{dist}(u,v) + \ell(u) )
  \end{equation}
  where $\mathrm{dist}$ denotes the graph distance in $Q$.
\end{rmk}

\subsection{Mirror bijection}
\label{sec:mirror}
We describe here a ``mirror'' formulation of the bijection $\Phi$. Define a \emph{mirror-well-labelled hypermap}
as a vertex-labelled hypermap $H$ such that the ``mirror'' of $H$ is well-labelled. 
More precisely, 
for any edge $e=\{v,u\}$ ---with the dark face incident to $e$ on the \emph{left} of $e$ traversed from $v$ to $u$---
we have $\ell(u)\geq \ell(v)-1$. Denote by $\mH$ the family of mirror-well-labelled hypermaps. 
The notions of right-neighbor, right local min and right local max are defined in the same way as for
the family $\cH$ (we do not take a mirror definition). 
For a dark face $f$ of $H\in\mH$, define the \emph{ccw-type} of $f$ as the cyclic list 
of labels of the vertices in counterclockwise order around $f$. 
Let $B\in\cB$. Place a vertex $v_f$ inside each face $f$ of $B$. 
Then apply the so-called \emph{complementary rules}
(shown in Figure~\ref{fig:rules} right-part)  
in $f$, i.e., when traversing $f$ clockwise, for each ascending edge $e=\{u,v\}$, insert a new edge from $v_f$ to $u$ 
(note that the vertices of $B$ not incident to any of these new edges are exactly the local max of $B$). 
Then erase all the local max of $B$ and all edges of $B$. Call $S$ the obtained figure, which is a bipartite map, 
and $H$ the mirror-well-labelled hypermap having $S$ as star-representation.  Let $\Phi^-$ be the mapping that associates
$H$ to $B$. Let $\mathrm{opp}$ 
be the mapping (operating on any $B\in\cB$ or $H\in\cH$)  
that replaces the label of each vertex by its opposite (note that $\mathrm{opp}$ maps
$\cB$ to $\cB$ and $\cH$ to $\mH$). 
In fact, given the fact that the complementary rules are the opposite of the BDG rules, we have
$\Phi^-=\mathrm{opp}\circ \Phi\circ \mathrm{opp}$.  Hence, as a consequence of Theorem~\ref{theo:bijmobile} we obtain:

\begin{coro}\label{coro:bij_anti_mobile}
The mapping $\Phi^-$ is a bijection  between $\cB$ and $\mH$. 
For $B\in\cB$ and $H=\Phi^-(B)$, each vertex $v$ of $H$ corresponds to a non local max vertex 
$v'$ of $B$ of the same label, and $v$ is a right local min in $H$ iff $v'$ is a local min in $B$, 
each face $f$ of $H$ corresponds to a local max vertex of $B$ of label $\maxi(f)+1$, 
and each dark face of $H$ of ccw-type $\tau$
corresponds to a face of $B$ of ccw-type $\uctau$.
\end{coro}

\subsection{Mobiles, pointed bipartite maps, and pointed hypermaps}
\label{sec:mobspec}
Define a \emph{mobile} as the star-representation of a well-labelled hypermap with a unique light face and
with minimal label $1$. In other words a mobile
is a bipartite plane tree with black unlabelled vertices and white labelled vertices, such that the minimal label is $1$ 
and for any two consecutive neighbors $v,u$ in clockwise order around a black vertex, 
$\ell(u)\geq \ell(v)-1$. The following definitions are inherited from the concepts in
the associated hypermap with a unique light face.  
For any black vertex $b$ in a mobile, define the \emph{cw-type} of $b$
as the  cyclic sequence given by the labels of the neighbors in clockwise order around $b$.
A white vertex $v$ is called a \emph{right neighbor} of a white vertex $u$, if 
$u$ and $v$ are consecutive in counterclockwise order around a black vertex (which is their
unique common neighbor). And a white vertex $u$ is called a \emph{right local min} 
(resp. \emph{right local max}) in the mobile if any right neighbor $v$ of $u$
satisfies $\ell(v)\geq\ell(u)$ (resp. $\ell(v)\leq\ell(u)$). 

\begin{claim}\label{claim:geodB}
Let $B$ be a bipartite map with a pointed vertex $v$. Then there is a unique suitable labelling of the vertices of $B$ such that 
 $v$ is the unique local min, and $\ell(v)=0$. This labelling is the \emph{geodesic labelling} with respect to $v$,
 i.e.\ such that $\ell(u)$ is the graph distance $\mathrm{dist}(v,u)$ from $v$ to $u$. 
\end{claim}
\begin{proof}
It is clear that the geodesic labelling satisfies these properties. Conversely, for any labelling satisfying these properties, each vertex $u$ has a path to $v$ that decreases in label, hence of length $\ell(u)$.
So $\ell(u)\geq \mathrm{dist}(v,u)$. Moreover, for any suitable labelling one has trivially
$\ell(u)\leq \mathrm{dist}(v,u)$ (the sequence of labels along a geodesic path from $v$ to $u$
increases by at most $1$ at each edge).  Hence $\ell(u)=\mathrm{dist}(v,u)$ for any vertex $u$. 
\end{proof}

\begin{figure}
\begin{center}
\includegraphics[width=13cm]{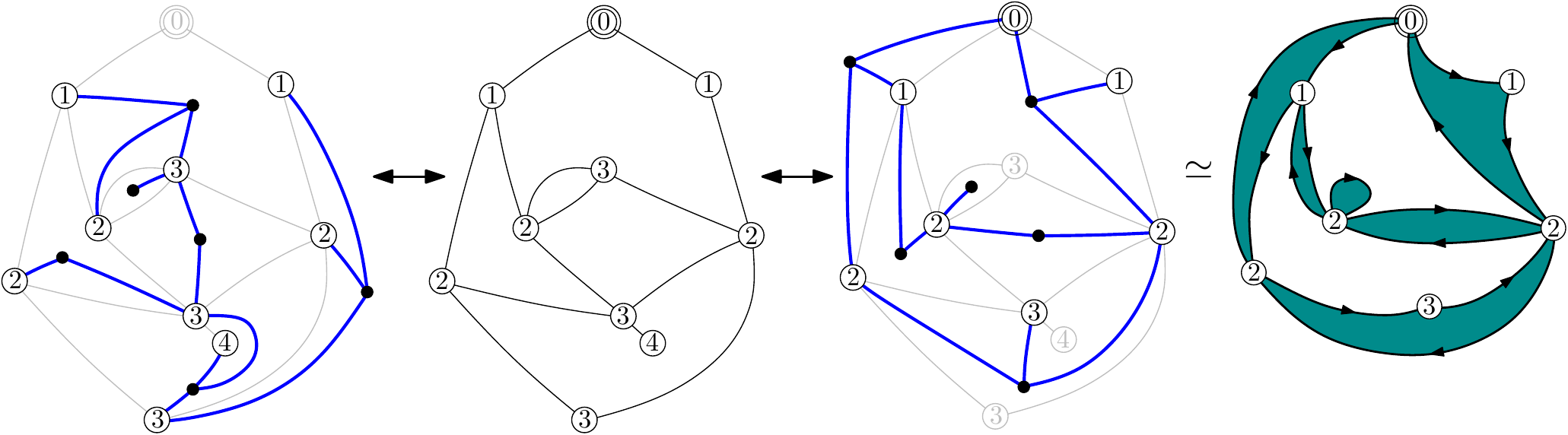}
\end{center}
\caption{A vertex-pointed map endowed with its geodesic labelling, 
on its left the associated mobile, on its
right the associated vertex-pointed hypermap (endowed with its geodesic labelling).}
\label{fig:bijpointed}
\end{figure}

Applying Theorem~\ref{theo:bijmobile} to the subfamily of $\cB$ with a unique local min vertex $v$ and with $\ell(v)=0$, we recover
the following result from~\cite{BDG04} (see the left part of Figure \ref{fig:bijpointed} for an illustration):

\begin{prop}\label{prop:point_bip_mobiles}
There is a bijection between vertex-pointed bipartite maps and mobiles with the following properties.
For $B$ a vertex-pointed bipartite map, endowed
with its geodesic labelling, and $M$
the corresponding mobile, each white vertex $v$ of $M$ corresponds to a non-pointed vertex $v'$ 
of $B$ of the same label, and $v$ is a right local max of $M$ iff $v'$ is a local max of $B$.  
And each black vertex of $M$ of cw-type $\tau$ corresponds to a face of $B$ of cw-type $\lctau$. 
\end{prop}

Given an hypermap $H$, define the \emph{canonical orientation} of $H$ as the orientation where each edge is directed so as to 
have its incident dark face on its right. If $H$ has a pointed vertex $v$, 
the \emph{geodesic labelling} of $H$ with 
respect to $v$ is the labelling of vertices where $\ell(u)$ gives the length of a shortest directed path 
(in the canonical orientation) from $v$ to $u$. Similarly as in Claim~\ref{claim:geodB} (and with the same proof arguments) we have:

\begin{claim}\label{claim:geodH}
Let $H$ be an hypermap with a pointed vertex $v$. Then there is a unique mirror-well-labelling of the vertices of $H$ such that 
 $v$ is the unique right local min, and $\ell(v)=0$. This labelling is the geodesic labelling with respect to $v$. 
\end{claim}

Applying Corollary~\ref{coro:bij_anti_mobile} to the subfamily of $\mH$ with a unique local min vertex $v$ and with $\ell(v)=0$, we obtain
the following result (see the right part of Figure \ref{fig:bijpointed} for an illustration):

\begin{prop}\label{prop:bij_bip_hyp}
There is a bijection between vertex-pointed bipartite maps and vertex-pointed hypermaps with the following properties. 
For $B$ a vertex-pointed bipartite map endowed
with its geodesic labelling, and $H$
the corresponding vertex-pointed hypermap endowed with its geodesic labelling, 
each light face $f$ of $H$ corresponds to a local max vertex of $B$ of label $\maxi(f)-1$,
 each vertex of $H$
corresponds to a non local max vertex of $B$ of the same label (so the pointed vertices correspond to each other), and 
 each dark face of $H$ of ccw-type $\tau$ corresponds to a face of $B$ of ccw-type $\uctau$. 
\end{prop}

\begin{figure}
\begin{center}
\includegraphics[width=12cm]{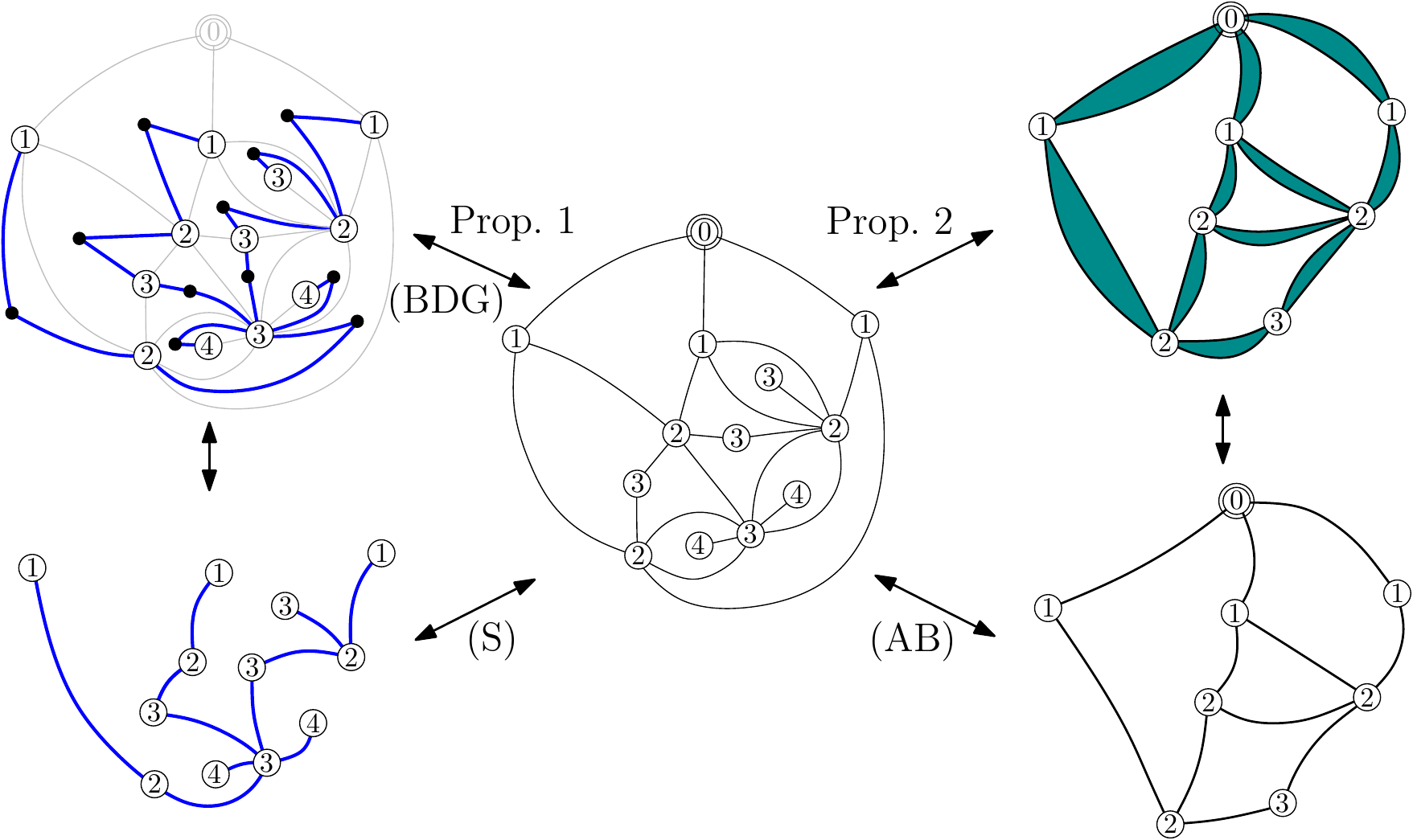}
\end{center}
\caption{In the case where the intermediate suitably labelled map is a quadrangulation, 
the mobile obtained by the BDG rules (bijection in~\cite{BDG04}, recovered in Proposition~\ref{prop:point_bip_mobiles}) 
simplifies to a well-labelled tree (with minimal label $1$), which corresponds
to Schaeffer's bijection~\cite{ChSc04,SchPhD} (symbol S in the diagram). 
Moreover, the pointed hypermap obtained by the complementary rules (Proposition~\ref{prop:bij_bip_hyp})
simplifies to a map, which corresponds to the bijection of Ambj\o rn and Budd~\cite{AmBudd} 
(symbol AB
in the diagram) between
pointed quadrangulations and pointed maps.
}
\label{fig:bij_ambjorn_budd}
\end{figure}

As shown in Figure~\ref{fig:bij_ambjorn_budd} (right part), 
in the case where all dark faces of the hypermap have degree~$2$, 
we recover the bijection of Ambj\o rn and Budd between pointed maps and
pointed quadrangulations.  

\begin{figure}
\begin{center}
\includegraphics[width=12cm]{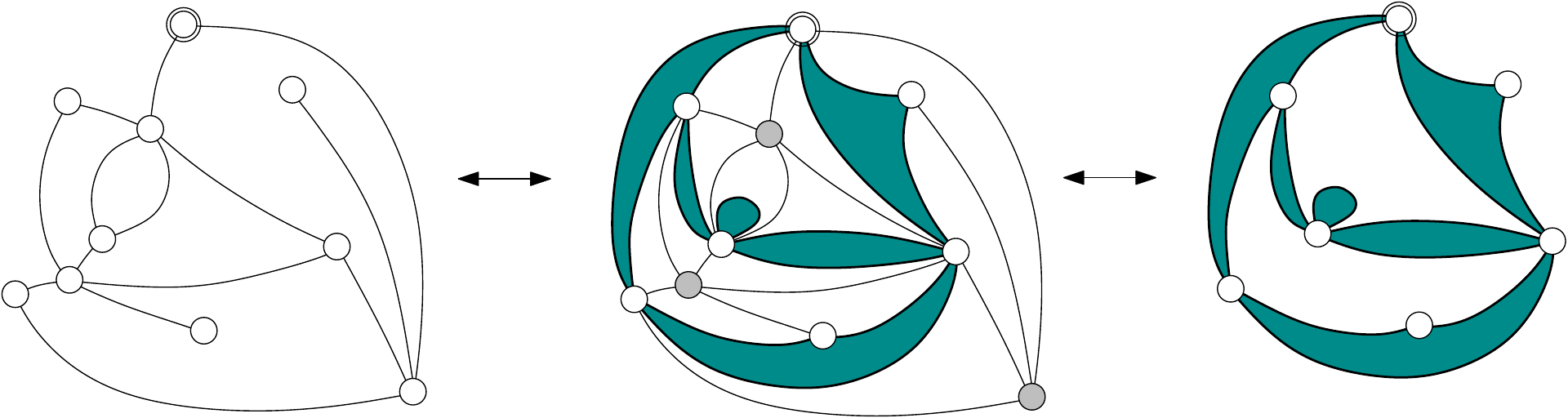}
\end{center}
\caption{The usual bijection between pointed bipartite maps and pointed hypermaps.
From left to right: color white (resp. gray) the vertices that are at even (resp. odd) distance
from the pointed vertex, then inside each face $f$ of the bipartite map insert a dark face that connects
the white corners of $f$. From right to left: insert a gray vertex $v_f$ inside each light 
face $f$ of the hypermap, and connect $v_f$ to all corners around $f$.}
\label{fig:bij_hypermap_classique}
\end{figure}

\begin{rmk}
There is already a classical bijection between (vertex-bicolored) 
bipartite maps and hypermaps
such that each face of degree $2s$ in the bipartite map corresponds to a dark face of degree $s$ in the corresponding
hypermap, see Figure~\ref{fig:bij_hypermap_classique}
(note that the bijection of Figure~\ref{fig:bij_hypermap_classique} is the application of $\Phi$ with the special suitable
labelling where every vertex at even distance from the pointed vertex has label $1$ and every 
vertex at odd distance has label $0$). 
 However, this bijection does not have the distance conservation property of 
Proposition~\ref{prop:bij_bip_hyp}. 
\end{rmk}

Composing the bijections of Propositions~\ref{prop:point_bip_mobiles} and~\ref{prop:bij_bip_hyp} we obtain (see Figure~\ref{fig:bijpointed}):

\begin{theo}\label{theo:hyp_mobiles}
There is a bijection between vertex-pointed hypermaps and mobiles with the following properties.  
For $H$ a vertex-pointed hypermap endowed with its
geodesic labelling, and $M$ the corresponding mobile, each face $f$ of $H$ 
corresponds to a right local max of $M$ of label $\maxi(f)+1$, 
each unpointed vertex of $H$ corresponds to a non right local max 
vertex of $M$ of the same label   
and each dark face of $H$ of ccw-type $\tau$ corresponds to a black vertex of $M$ of cw-type  
$\ubtau$.

In particular, for $p \geq 2$, the bijection restricts to a bijection
between vertex-pointed \emph{$p$-hypermaps} (hypermaps with all dark
faces of degree $p$) and \emph{$p$-mobiles} (mobile with all black
vertices of degree $p$).
\end{theo}

Note that there is already a bijection in~\cite{BDG04} between vertex-pointed hypermaps and certain labelled decorated (multitype) 
plane trees. The bijection of~\cite{BDG04} (which also relies on the geodesic labelling of the 
pointed hypermap) 
 has the advantage that it keeps track of the degrees of light faces, an information
that is lost with the bijection of Theorem~\ref{theo:hyp_mobiles}. However the price to pay 
is that the decorated trees in~\cite{BDG04} are far more complicated than the mobiles of  Theorem~\ref{theo:hyp_mobiles}. 

\subsection{Specialization to constellations}
\label{sec:conspec}
For $p\geq 2$, a (planar) \emph{$p$-constellation} is a $p$-hypermap with
 all light faces of degree a multiple of $p$. Constellations are also characterized as $p$-hypermaps
that can be vertex-colored, with colors $0,1,\ldots,p-1$, such that in clockwise order
around any dark face the colors are $0,1,\ldots,p-1$ (seen this way they correspond
to certain factorizations, into $p$ factors, in the symmetric group).  
Equivalently, given a vertex-pointed hypermap $H$ endowed with its geodesic labelling, 
$H$  is a $p$-constellation iff the labels modulo $p$ of the vertices
in clockwise order around each dark face are $0,1\ldots,p-1$.  
A \Lu cyclic sequence of length $r$ is said to be \emph{descending}
if it has a unique rise and $r-1$ descents (by $1$). Define a 
\emph{$p$-descending mobile} as a $p$-mobile with all black vertices of descending cw-type. 
From the discussion above, a $p$-hypermap $H$, endowed with its geodesic labelling, 
is a $p$-constellation iff all its dark faces are of descending ccw-type.   
Hence, as a specialization of Theorem~\ref{theo:hyp_mobiles} we obtain:

\begin{prop}\label{prop:const1}
There is a bijection between vertex-pointed $p$-constellations and $p$-descending mobiles, with 
the same parameter-correspondence as in Theorem~\ref{theo:hyp_mobiles}. 
\end{prop}

The bijection is shown  in Figure~\ref{fig:bij_const} for $p=3$ (forgetting the bottom-left drawing for the moment) and
 in Figure~\ref{fig:bij_bipartite} 
for $p=2$ 
(constellations identify to bipartite maps by shrinking each dark face of degree $2$ into an edge and $2$-descending mobiles
identify to suitably labelled plane trees -- with minimum label $1$ -- by erasing the black vertices).

\begin{figure}
\begin{center}
\includegraphics[width=12cm]{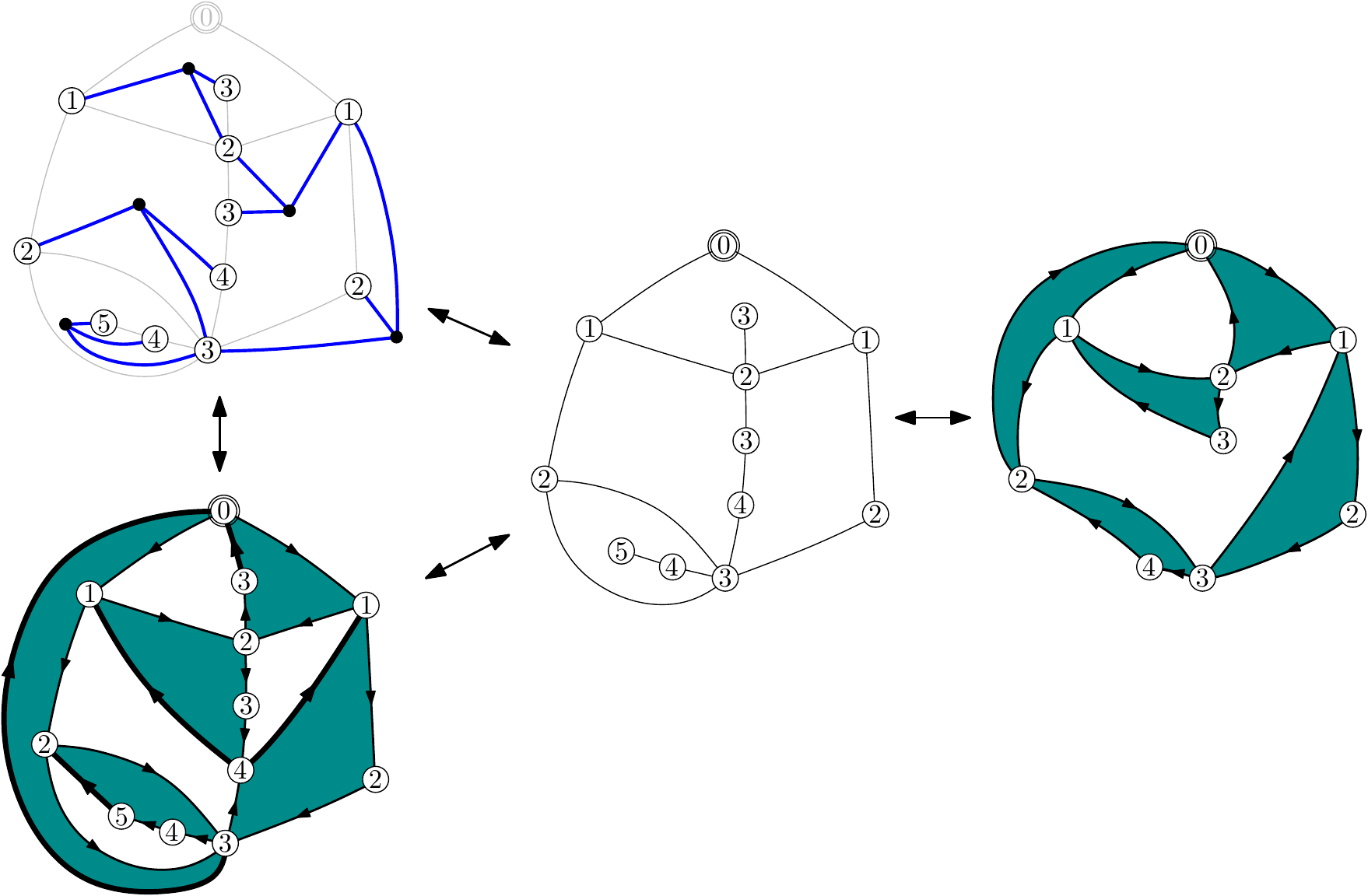}
\end{center}
\caption{Middle: a pointed stretched $6$-angulation. Right: the associated
(by the complementary rules) pointed 3-constellation. Top-left: the associated (by the BDG rules)
mobile. Bottom-left: the $4$-regular constellation obtained by drawing a diagonal
in each (stretched) face from the largest to the smallest vertex. Upon composing these elementary
bijections, we obtain propositions \ref{prop:const1} (connecting top-left to right) and \ref{prop:const2}
(connecting bottom-left to right).
}
\label{fig:bij_const}
\end{figure}

\begin{figure}
\begin{center}
\includegraphics[width=12cm]{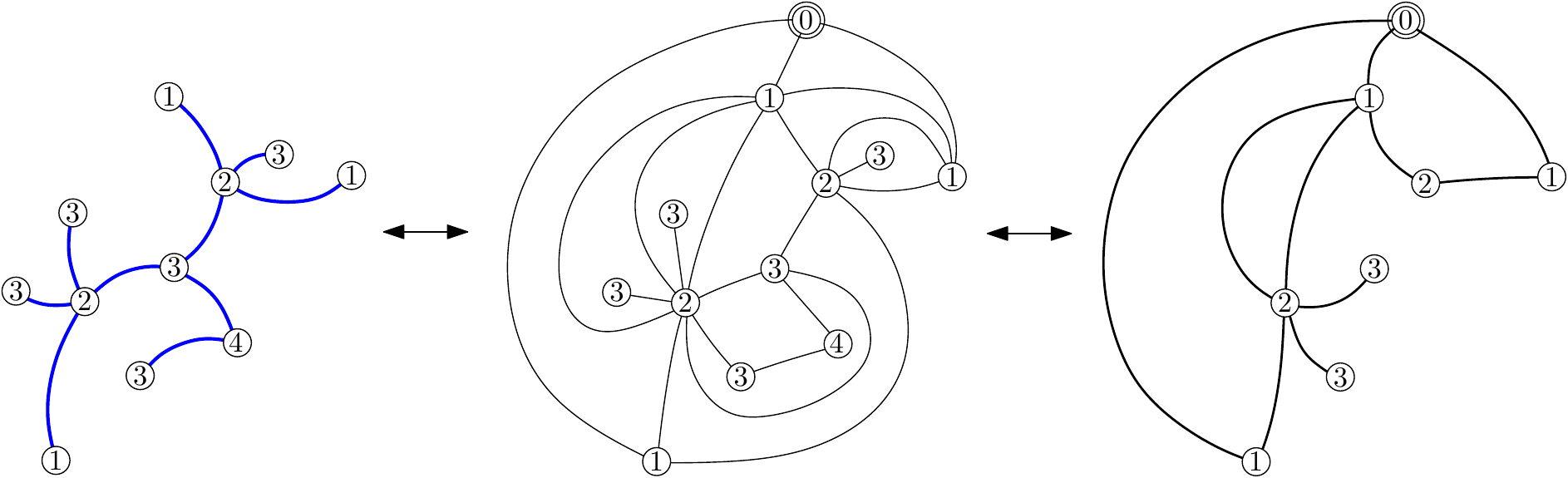}
\end{center}
\caption{In the case $p=2$, a $p$-constellation identifies to a bipartite map $G$ (on the right side),
and a $p$-descending mobile identifies to a suitably labelled plane tree $T$ 
with minimal label $1$ (left side). Each face of $G$ corresponds to a local max of $T$, 
each unpointed vertex of $G$
corresponds to a non local max vertex of $T$ of the same label, and each edge of $G$ of 
type $(i,i-1)$ corresponds to an edge of $T$ of type $(i+1,i)$.  
The bijection is a specialization
of the bijection of Ambj\o rn and Budd using stretched pointed quadrangulations
as intermediate objects (middle).
}
\label{fig:bij_bipartite}
\end{figure}

Let $B$ be a pointed bipartite map endowed with its geodesic labelling.
A face $f$ of $B$, of even degree $2s$, is said to be \emph{stretched}  
if its  type (cw or ccw)  is of the 
form $i,i+1,\ldots,i+s-1,i+s,i+s-1,\ldots,i+1$, i.e.,  there are $s$ rises followed by $s$ descents.
And $B$  is called \emph{stretched} if all its faces are stretched.
Stretched pointed $2p$-angulations are the bijective intermediates 
in Proposition~\ref{prop:const1} (i.e.\ these are the maps in correspondence with $p$-descending mobiles
via $\Phi$ and with vertex-pointed $p$-constellations via $\Phi^{-}$). Moreover, as shown in~\cite{BDG04}, $p$-descending mobiles 
are in bijection with vertex-pointed $(p+1)$-constellations with all light faces of degree $p+1$
(shortly called $(p+1)$-regular constellations -- note that those are nothing but
Eulerian $(p+1)$-angulations endowed with a proper bi-coloring of their faces).  
We provide here a simple shortcut (to jump over $p$-descending mobiles) in  
the bijective chain:

\vspace{.2cm}

pointed $p$-constellations $\leftrightarrow$ stretched pointed $2p$-angulations \\
\phantom{1}\hspace{.1cm} $\leftrightarrow$ $p$-descending mobiles $\leftrightarrow$ pointed $(p+1)$-regular constellations

\vspace{.2cm}

Given a pointed stretched $2p$-angulation $B$, draw in each (stretched) face a diagonal $e$ 
from the largest to the smallest vertex. This splits the face into two faces of degrees $p+1$, 
and we color the one on the right of $e$ as dark and the one on the left of $e$ as light. 
 The obtained figure is clearly a pointed $(p+1)$-regular constellation $E$. In addition,
if $B$ is endowed with its geodesic labelling, then the induced labelling on $E$
is exactly the geodesic labelling of $E$. The inverse mapping is easy.
Given a $(p+1)$-regular constellation $E$, endow $E$ with its geodesic labelling, 
which has the property that the labels in clockwise (resp. counterclockwise) 
order around each dark (resp. light) face are of the form $i,i+1,\ldots,i+p$. 
Erasing the edges of the form $i,i+p$, we naturally obtain a stretched 
(vertex-pointed) bipartite $2p$-angulation
endowed with its geodesic labelling. To summarize, we obtain:

\begin{prop}\label{prop:const2}
There is a bijection between vertex-pointed $p$-constellations and vertex-pointed $(p+1)$-regular
constellations with the following properties.  For $C$ a pointed $p$-constellation and $E$ the corresponding pointed $(p+1)$-regular constellation (both endowed with their geodesic labelling), each  face $f$ of $C$ corresponds
to a right local max vertex of $E$ of label $\maxi(f)+1$, and each vertex $v$ of $C$ corresponds to a 
non right local max vertex $v'$ of $E$ of the same label.
\end{prop}

\begin{figure}
\begin{center}
\includegraphics[width=9cm]{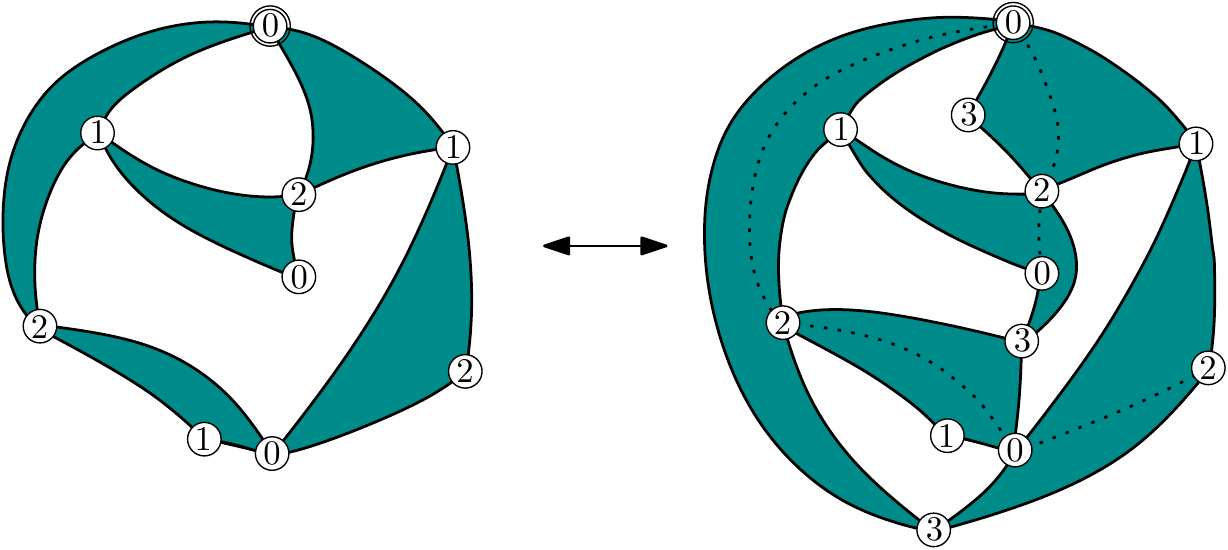}
\end{center}
\caption{The usual bijection between pointed $p$-constellations and pointed $(p+1)$-regular
constellations. From left to right: add a vertex of color $p$ in each light face and
extend each dark face by making the edge $\{p-1,0\}$ pass by the vertex in the light face to its left.
From right to left: draw a diagonal in each dark face between the vertices of color $0$ and $p-1$,
then erase the vertices of color $p$ and their incident edges.}
\label{fig:bij_const_classique}
\end{figure}

There is already a classical bijection between pointed $p$-constellations and pointed $(p+1)$-regular 
constellations, using the characterization in terms of vertex-coloring, see Figure~\ref{fig:bij_const_classique}.  
However, similarly as in Proposition~\ref{prop:bij_bip_hyp}, the bijection of Proposition~\ref{prop:const2} 
has the advantage that it preserves the distance to the pointed vertex. 

\begin{rmk} \label{rmk:2dmhyp} Using the trivial identification of
  $2$-constellations with bipartite maps, we may combine
  Proposition~\ref{prop:const1} and
  Proposition~\ref{prop:bij_bip_hyp}, to obtain a bijection between
  $2$-descending mobiles and vertex-pointed hypermaps. In this
  bijection, black vertices and right local max of a mobile
  correspond respectively to edges and dark faces of the associated hypermap.
\end{rmk}

\subsection{Proof that the mappings $\Phi/\Psi$ give a bijection}\label{sec:proof_bij}
\subsubsection{Proof of Claim~\ref{claim:Smap}.} \label{sec:proof_claim_Smap}
Let $S$ be the figure obtained from $B\in\cB$ by applying the BDG rules
(and then deleting the local min vertices 
and the edges of $B$). We show here that $S$ is a map, i.e., is connected. 
Let $s_v$, $s_e$, $s_f$ be the numbers of vertices, edges, and faces of $S$ (a face is a connected
component of the sphere from which  $S$ is cut out). Let $m$ be the number
of local min of $B$. By the BDG rules, $s_e$ is the number of edges of $B$, and $s_v+m$ is the 
number of vertices plus the number of faces of $B$. Hence, by the Euler relation applied to $B$, 
$s_v+m=s_e+2$. Moreover, the Euler relation applied to $S$ ensures that $s_v+s_f=s_e+1+k$, where $k$
is the number of connected components of $S$. Hence, $S$ is connected iff $s_f\leq m$ (indeed, $S$
is connected iff $k\leq 1$). To prove that $s_f\leq m$, it is enough to show that there is
at least one local min of $B$ inside each face of $S$. Let $C$ be the so-called completed map of $B$,
obtained from $B$ by adding a black vertex $v_f$ inside each face $f$ of $B$ and connecting
$v_f$ to all corners around $f$. 
Note that $C$ is a triangulation (each triangle is incident to two white vertices
and a black one) and that $S$ and $B$ 
are embedded subgraphs of $C$. 
To show that there is a local min of $B$ inside each face
of $S$, it is enough to show that from each vertex inside a (triangular) face $\tau$ of $C$
one can reach a local min vertex of $B$ without meeting $S$. 
In turn this reduces  
 to show the following claim, where the \emph{index} of a (triangular) face of $C$
is the smallest label over the (two) labelled vertices incident to $\tau$:

\begin{claim}\label{claim:accessible}
For any triangular face $\tau$ of $C$ such that $\tau$ is not incident to a local min of $B$, and
for any point $p$ in $\tau$, it is possible to move continuously, and without
meeting $S$, from $p$ to a point $p'$ that is in a triangular face $\tau'$ of index strictly smaller
than the index of $\tau$. 
\end{claim}
\noindent\emph{Proof of the claim.} The situation is shown in Figure~\ref{fig:local_prop}.
\begin{figure}
\begin{center}
\includegraphics[width=2.6cm]{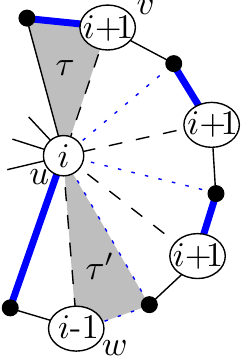}
\end{center}
\caption{Situation around $u$ in the proof of Claim~\ref{claim:accessible}; 
edges between two white vertices (never in $S$) are shown dashed,
edges from a black to a white vertex are shown dotted if not in $S$, bold if in $S$, solid if 
the status is not determined (by the BDG rules) from the labels in the figure.}
\label{fig:local_prop}
\end{figure}
Let $u$ be the vertex of smallest label incident to $\tau$, call $i$ the index of $u$, 
and let $v$ be the other labelled
vertex of $\tau$, of label $i+1$. 
Since $u$ is not a local min, there is a neighbor $w$ of $u$ such that, in clockwise order between $v$ and $w$, 
all labelled neighbors of $u$ have label $i+1$, and the 
vertex $w$ has label $i-1$. 
Hence, by the BDG rules, the edges from $u$ to any black vertex between $v$ and $w$ 
are not in $S$. Hence, denoting by $\tau'$ the face of $C$   having the edge $\{u,w\}$
(traversed from $u$ to $w$) to its left,  
  one can go from any point inside $\tau$ to any point inside $\tau'$ without meeting $S$.
This concludes the proof of the claim, since $\tau'$ has smaller index than $\tau$. \hfill $\square$

It follows from the proof that there is actually (since $s_f\leq m$ implies $s_f=m$ by the Euler relation)
\emph{exactly} one local min ---denoted $v_f$--- of $B$ inside each face $f$ of $S$. Moreover 
recall the BDG rule: ``for each edge $e=\{b,w\}\in S$, with $b$ the black 
extremity and $w$ the white (labelled) extremity, the next edge after $e$ in counterclockwise
order around $w$ leads to a white vertex of label $\ell(w)-1$". This implies 
that each white vertex $w$ on the contour of $f$ is incident to an edge of $B$ starting inside $f$ and leading
to a white vertex $w'$ of label $\ell(w)-1$. 
Necessarily $w'$ is inside $f$ (at $v_f$) or is on the contour of $f$. 
Hence, taking $w$ to have label $\mini(f)$, the neighbor $w'$ of smaller label is 
necessarily at $v_f$, which ensures that $\ell(v_f)=\mini(f)-1$.

\subsubsection{Proof of Claim~\ref{claim:Bmap} and that $\Phi\circ\Psi=\mathrm{Id}$.} 

Let $H\in \cH$, let $S$ be the star representation of $H$, and superimpose $S$ with $B:=\Psi(H)$. 
The situation around a black vertex of $S$ is shown in Figure~\ref{fig:psi_face}: each black vertex of $S$ yields
a face of $B$, and these faces 
cover the entire surface (sphere), so we conclude that $B$ is a map (a map can also be defined as a gluing
of topological disks to form a closed surface). In addition,
it is clear also from the figure that the edges of $S$ will exactly be those selected by the BDG rules.
Hence $\Phi(B)=H$. 

\begin{figure}
\begin{center}
\includegraphics[width=13cm]{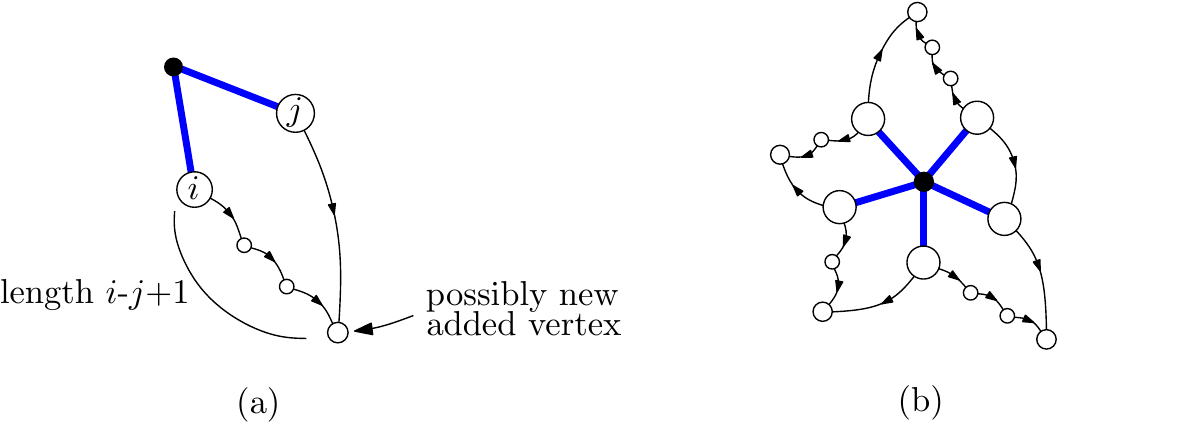}
\end{center}
\caption{(a) Situation (after applying $\Psi$) for a corner at a black vertex of $S$ (edges of $B$ are oriented in label decreasing direction). 
(b) Situation (after applying $\Psi$) around a black vertex of $S$: each black vertex of $S$ yields a face of $B$.}
\label{fig:psi_face}
\end{figure}

\subsubsection{Proof that $\Psi\circ\Phi=\mathrm{Id}$.} 
Let $B\in\cB$, $H=\Phi(B)$, and $S$ the star-representation of $H$. 
In the following it is convenient to see $B$ as superimposed with $S$.  
An easy observation (following from the BDG rules) 
is that, if we direct the edges of $B$ in label-decreasing way, then 
each corner $c$ of $S$ at a white vertex $v$ ---since $B$ and $S$ are superimposed, 
there is a bunch of edges
of $B$ in $c$--- contains a unique edge of $B$ going out of $v$, 
which is the clockwise-most in $c$ (note that this local property
is to be satisfied after applying $\Psi$), see Figure~\ref{fig:proof_e}(a).  
Denote by $e_c$ this edge. 
 We have seen in Section~\ref{sec:proof_claim_Smap} 
that $H\in\cH$, there is exactly one local min of $B$
strictly inside each face $f$ of $S$, and this local min has label $\mini(f)-1$. 
So the first step of $\Psi$ (adding a vertex inside each face $f$ of label $\mini(f)-1$) is the inverse
of the last step of $\Phi$ (deleting all the local min); and in addition for each corner $c$ of $S$ at a 
white vertex $v$, if $e_c$ goes to the local min in the face incident to $c$, then $e_c$
will be created by $\Psi$. It remains to show that for each corner $c$ such that $e_c$ does
not go to a local min (otherly stated, for each edge $e$ of $B$ not incident to a local min), 
 $e_c$ will be created by $\Psi$. 
Let $f$ be any face of $S$, with $v_f$ the local min of $B$ inside $f$, and let $e$ be an edge of $B$ 
 inside $f$ and 
not incident to $v_f$. 
Adding $e$ to $f$ splits $f$ into two faces 
 $L_f(e), R_f(e)$ respectively on the left and on the right of $e$ directed in label decreasing way. Let $u$
be the extremity of $e$ of largest label and $v$ the extremity with smaller label, say $\ell(u)=i$ and $\ell(v)=i-1$. 
To show that $e$ will be created when applying $\Psi$ (to $S$), it remains to establish   
 the following property:

\begin{claim}\label{claim:e}
The local min $v_f$ is inside $L_e(f)$, $v$ has a unique incident corner
 in $R_e(f)$ (the one delimited by $e$ on the right side) 
and any other corner 
in $R_e(f)$ at a white labelled vertex $w$ satisfies $\ell(w)\geq i$.  
\end{claim}
\noindent\emph{Proof of the claim.} 
Recall the BDG rule (illustrated in Figure~\ref{fig:proof_e}): 
for each edge $e=\{b,w\}\in S$, with $b$ the black 
extremity and $w$ the white (labelled) extremity, the next edge after $w$ in counterclockwise
order around $w$ leads to a white vertex of label $\ell(w)-1$. 
Let $w_{\ell}$ be a white vertex of smallest possible label on the contour of $L_e(f)$, note that $w_{\ell}\neq u$.
Since $w_{\ell}\neq u$, the BDG rule recalled above implies that $w_{\ell}$ has a neighbor of label $\ell(w_{\ell})-1$
in $L_e(f)$, either at $v_f$ or on the contour of $L_e(f)$. The second case is excluded by minimality of $w_{\ell}$,
so we conclude that $v_f$ is inside $L_e(f)$, see Figure~\ref{fig:proof_e}(b). 
The statement about $R_e(f)$ is proved similarly. 
Denote by $c$ the corner of $R_e(f)$ that is incident to $v$ and delimited by $e$ on the right side. 
Call a corner of $R_e(f)$ admissible
if it is different from $c$ and incident to a white vertex.  
Choose an admissible corner $c_0$ 
in $R_e(f)$ of smallest possible label (label of the incident white vertex). Let $w_0$ be the white vertex incident to $c_0$.  
Again, since $c_0\neq c$ the BDG rule implies that $w_0$ has a neighbor of label $\ell(w_0)-1$ 
on the contour of $R_e(f)$.
This neighbor is necessarily $v$ (otherwise it would yield an admissible corner 
of smaller label than $c_0$, contradicting
the minimality of $c_0$), hence $\ell(w_0)=i$.  \hfill $\square$

\begin{figure}
\begin{center}
\includegraphics[width=10cm]{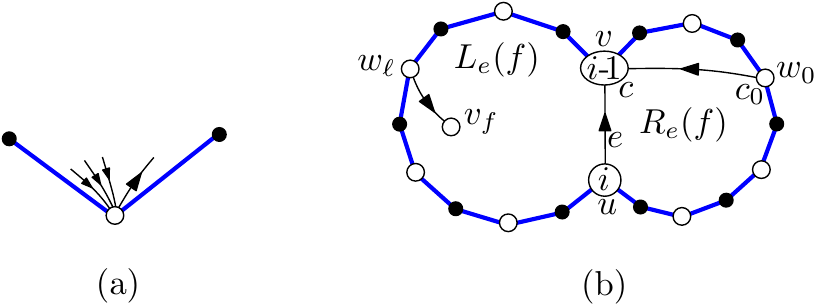}
\end{center}
\caption{(a) After applying $\Phi$ to $B$, 
each corner $c$ of $S$ contains a unique edge of $B$ going out, which 
is the clockwise-most edge of $B$ inside $c$. 
(b) Situation in the proof of Claim~\ref{claim:e}. 
The shown edges of $B$ are
 directed in the label-decreasing way.} 
\label{fig:proof_e}
\end{figure}

\subsection{Extension to higher genus}
\label{sec:highergenus}
Up to now the bijective results have been stated and proved with maps and hypermaps 
on the sphere, but everything 
can be defined, stated and proved in the same way as in genus $g\geq 0$ (the sphere
corresponds to $g=0$, recall that a map of genus $g$ is a connected graph $G$ embedded
on the genus $g$ surface $\Sigma$ such that any component of $\Sigma\backslash G$ is 
a topological disk). We point out here the few places where one can not copy verbatim. 

In the proofs, the only place 
involving $g$ as a parameter
is in the use of the Euler relation in the proof of Claim~\ref{claim:Smap}. In genus $g$, 
 the Euler relation applied to $B$ gives 
$s_v+m=s_e+2-2g$, and the Euler relation applied to $S$ ensures that $s_v+s_f\geq s_e+2-2g$, 
with equality iff $S$ is a map of genus $g$. Hence $S$ is a map of genus $g$ iff $s_f\leq m$, as in the genus $0$
case.

In the definitions, note that a mobile (i.e.,  the star-representation of a  well-labelled 
hypermap with a unique light face and minimal label $1$) is not a tree in higher genus
 but a unicellular map of genus $g$.   
And the notion of constellation  in genus $g$ has to be defined in terms of 
the color-property (in higher genus, the condition of light face-degrees being multiple of $p$ is
strictly weaker than the condition with colors), that is,  
for $p\geq 2$, a \emph{$p$-constellation} in genus $g$ is a $p$-hypermap of genus $g$ 
whose vertices can be colored, with colors in $0,1,\ldots,p-1$, such that the colors of vertices in 
clockwise order around any dark face are $0,1,\ldots,p-1$. 
The extension of Proposition~\ref{prop:point_bip_mobiles}
(bijection between pointed bipartite maps and mobiles) to higher genus was 
first given in \cite{Chapuy2009CPC}.

\section{Two-point functions depending on a single size parameter}
\label{sec:2psing}

In this section we shall use specializations of the above bijections
to compute distance-dependent two-point functions for a number of
families of maps or hypermaps.  More precisely, we shall concentrate
on $2$-hypermaps, i.e.\ hypermaps having all their dark faces of
degree $2$, and \emph{$3$-hypermaps} with all their dark faces of
degree $3$.  Recall that \emph{general maps}, whose two-point function
was already obtained in \cite{AmBudd}, are trivially identified
with $2$-hypermaps by blowing each edge of the map into a dark face of
degree $2$. We shall also consider the case of $2$-constellations,
trivially identified with \emph{bipartite maps}, and that of
\emph{$3$-constellations}.  All these maps will be counted according
to their natural size parameter, namely the number of edges (for maps)
or dark faces (for hypermaps).

The (distance-dependent) \emph{two-point function} of a class of maps
is, informally speaking, the generating function of such maps with two
marked points at a prescribed distance. More precisely, we consider
planar (hyper)maps that are both vertex-pointed and rooted, i.e.\ with
a marked oriented edge (the \emph{root edge}). In the case of
hypermaps, we furthermore assume that the root edge is oriented in
such a way that its incident dark face (the \emph{root face}) lies on
its right. The two-point function is defined as the generating
function of these vertex-pointed and rooted (hyper)maps with
prescribed geodesic distances from the pointed vertex to all vertices
incident to the root edge (for maps) or to the root face (for
hypermaps). Since we consider only a single size parameter, the
two-point function depends on a single formal variable $t$, the weight
per edge or dark face.

The reader shall be warned that, in order to avoid introducing too
many symbols, indexes or subscripts, we will keep the same notations
for the two-point function and related quantities regardless of the
class of maps considered.

\subsection{The two-point function of general maps}
\label{sec:2pgen}

Consider a pointed rooted map (i.e.\ a map with a pointed vertex and a marked oriented edge): we say that its 
root edge is of type $(k,j)$ if its origin and endpoint are at respective distances $k$ and $j$ from the pointed vertex.  
Here $k$ and $j$ are non-negative integers satisfying $|j-k|\leq 1$. Following the notations of \cite{BG12}, we denote by 
$R_i\equiv R_i(t)$, $i\geq 1$, the generating function of pointed rooted maps \emph{enumerated with a weight} 
$t$ \emph{per edge} whose root edge is of type $(j-1,j)$ for $j\leq i$.  By reversing the orientation of the root edge, pointed 
rooted maps whose root  edge is of type $(j+1,j)$ for $j\leq i$ are enumerated by $R_{i+1}$. We finally denote by 
$S_i^2\equiv S_i(t)^2$,  $i\geq 0$, the generating function of pointed rooted maps whose root edge 
is of type $(j,j)$ for $j\leq i$ (we write this generating function as a square to stick to the notations of \cite{BG12}).

\subsubsection{Computation from the bijective approach}
Using the trivial identification of maps with $2$-hypermaps, $R_i$ may alternatively be understood as the generating 
function for vertex-pointed and rooted $2$-hypermaps, with a weight $t$ per dark face, and a root face of ccw-type 
$\tau=(j,j-1)$ for $j\leq i$ if we endow the hypermap with its geodesic labelling. From Theorem~\ref{theo:hyp_mobiles}, those 
are in one-to-one correspondence with $2$-mobiles, i.e.\ mobiles having all their black vertices of degree $2$,  
with a marked black vertex (in correspondence with the root face) of cw-type $\ubtau=(j,j+1)$, $j\leq i$.

Similarly, $S_i^2$ may be viewed as the generating function for vertex-pointed and rooted $2$-hypermaps, with a weight $t$ per dark face, 
and a root face of ccw-type $\tau=(j,j)$ for $j\leq i$ (under the geodesic labelling), one side of the root face being distinguished
(in correspondence with the root edge).
From Theorem~\ref{theo:hyp_mobiles}, those are now in in one-to-one correspondence with {$2$-mobiles} with a marked black 
vertex of cw-type $\ubtau=(j+1,j+1)$, $j\leq i$, one of the incident half edges being distinguished.

Recall that mobiles are required to have a minimal label $1$. The summation over all $j\leq i$ allows
however to waive this constraint. Indeed, shifting the labels by $i-j\geq 0$ in a mobile with a marked
black vertex of cw-type  $\ubtau=(j,j+1)$ (resp.  $\ubtau=(j+1,j+1)$) produces a new labelled tree with marked  
black vertex of cw-type  $\ubtau=(i,i+1)$ (resp.  $\ubtau=(i+1,i+1)$) whose minimal label $1+i-j$ is now, 
after summation over $j$, an arbitrary integer between $1$ and $i$ (respectively $i+1$). We shall call
\emph{floating} mobiles these new objects with an arbitrary positive minimal label (the rules for labels around
a black vertex remain unchanged).

Denoting by $T_i=T_i(t)$ the generating function of floating $2$-mobiles planted at a white vertex labelled $i$, enumerated
with a weight $t$ per black vertex, we deduce immediately the identifications
\begin{equation} \label{eq:genrsti}
\begin{split}
&R_i=1+t \,T_i T_{i+1}&\qquad i\geq 1,\\
&S_i =\sqrt{t}\, T_{i+1}&\qquad i\geq 0,
\end{split}
\end{equation}
where we incorporated in $R_i$ a conventional term $1$.
Now we have at our disposal explicit expressions for the generating function $T_i$, as obtained for instance 
by solving the equation
\begin{equation} \label{eq:gentrec}
T_i=\frac{1}{1-t\,(T_{i-1}+T_i+T_{i+1})}, \qquad i\geq 1
\end{equation}
with initial condition $T_0=0$ (this equation simply expresses the recursive nature of planted floating $2$-mobiles,
namely that a floating $2$-mobile planted at a white vertex labelled $i$ may be viewed as a sequence of
planted floating $2$-mobiles -- attached to the root vertex via bivalent black vertices -- having themselves a root label $i-1$, $i$ or $i+1$). 
From various techniques, it was found that \cite{GEOD}
\begin{equation}
\begin{split}
&T_i=T \frac{(1-y^i)(1-y^{i+3})}{(1-y^{i+1})(1-y^{i+2})} \\
&\text{where}\ \ T=1+3 t\,T^2\ \ \text{and}\ \ y+\frac{1}{y}+1=\frac{1}{t\, T^2}\ .
\label{eq:2pgenTi}
\end{split}
\end{equation}
Here, and for similar equations below, we always pick for $T$ the solution satisfying $T=1+O(t)$ and for 
$y$ the solution with modulus less than $1$.  
Plugging this formula in the expressions for $R_i$ and $S_i$ above, we deduce, after simplification, that
\begin{equation}
\begin{split}
&R_i=R \frac{(1-y^{i+1})(1-y^{i+3})}{(1-y^{i+2})^2}\\
&\text{where}\ \ R=1+ t\,T^2
\end{split}
\label{Rigm}
\end{equation}
and
\begin{equation}
\begin{split}
&S_i=S \frac{(1-y^{i+1})(1-y^{i+4})}{(1-y^{i+2})(1-y^{i+3})}=S-\sqrt{R\, y}\left(\frac{1-y^{i+2}}{1-y^{i+3}}-\frac{1-y^{i+1}}{1-y^{i+2}}\right)\\
&\text{where}\ \ S=\sqrt{t}\, T\ .
\end{split}
\label{Sigm}
\end{equation}
Note that $R=\lim_{i\to \infty}R_i$ and $S^2=\lim_{i\to \infty}S_i^2$ may be understood as the generating functions 
for pointed rooted maps whose root edge is of type $(j-1,j)$ and $(j,j)$ respectively \emph{without bound on $j$}.
In particular, the generating function for pointed rooted planar maps is $2(R-1)+S^2=3 t\,T^2=T-1$ which
(from the Schaeffer bijection for instance) is known to be half the generating function for pointed rooted 
quadrangulations. This result could have been deduced directly from the ``trivial'' bijection between general maps 
and quadrangulations.

Let us finally compute the generating function $V_i$ for vertex-pointed maps, enumerated with a weight $t$ per edge,
 with an extra marked vertex at distance $j\leq i$ from the pointed vertex, with $i\geq 1$. Recall that such doubly-pointed maps
(supposedly drawn on the sphere)  may present a $k$-fold symmetry by rotation around their two marked vertices (supposedly 
drawn at antipodal positions). As customary, we decide to enumerate maps with this $k$-fold symmetry with a 
\emph{symmetry factor} $1/k$. Only with this definition has $V_i$ a simple expression. As before,
$V_i$  may alternatively be understood as the generating function (with symmetry factors) for vertex-pointed  $2$-hypermaps wih 
an extra marked vertex labelled $j\leq i$ under the geodesic labelling. From Theorem~\ref{theo:hyp_mobiles}, these are in
one-to-one correspondence with $2$-mobiles with a marked non-right-local-max vertex of the same label $j$. 
Shifting the labels by $(i-j)$ and summing over all $j\leq i$, $V_i$ is the generating function for floating $2$-mobiles
with a marked non-right-local-max vertex labelled $i$.   
This allows to write
\begin{equation}
\begin{split}
&V_i=\sum_{k\geq 1}\frac{\left(t\,(T_{i-1}+T_{i}+T_{i+1})\right)^k}{k}-\sum_{k\geq 1}\frac{\left(t\,(T_{i-1}+T_{i})\right)^k}{k}\\
&\ \ =\log\left(\frac{1-t\, (T_{i-1}+T_{i})}{1-t\,(T_{i-1}+T_{i}+T_{i+1})} \right)\\
&\ \ =\log\left(T_i(1-t\, (T_{i-1}+T_{i})\right)\\
&\ \ =\log\left(1+t\, T_{i}T_{i+1} \right)\\
&\ \ =\log(R_i)\ ,
\end{split}
\end{equation} 
valid for $i\geq 1$ (the subtracted term in the first line was introduced to remove configurations were
the root vertex would be a right-local max).  
 
\subsubsection{Comparison with the continued fraction approach}
\label{sec:2pgencf}
It is interesting to compare the explicit expressions \eqref{Rigm} and \eqref{Sigm} to those obtained 
from the continued fraction approach developed in \cite{BG12} for maps with a control of their face degrees. 
Note that enumerating maps with a weight $t$ per edge is equivalent to enumerating maps 
\emph{with unbounded face degrees} and a weight $g_k=t^{k/2}$ per face of degree $k$.  The continued fraction approach allows to write
\begin{equation} \label{eq:2pgencontfracform}
\begin{split}
& R_i=R \frac{u_i u_{i+2}}{u_{i+1}^2}\\
& S_i=S-\sqrt{R} \left(\frac{\tilde{u}_{i+2}}{u_{i+2}}-\frac{\tilde{u}_{i+1}}{u_{i+1}}\right)
\end{split}
\end{equation}
where $R$ and $S^2$ have the same interpretation as above as pointed rooted map generating functions
and where $u_i$ and $\tilde{u}_i$ may be expressed in terms of $(i+1)\times (i+1)$ Hankel determinants $H_i$ 
and $\tilde{H}_i$ ($u_i=H_{i-2}/R^{(i-1)(i-2)/2}$ and $(i-1)S u_i-\sqrt{R} \tilde{u}_i=\tilde{H}_{i-2}/R^{(i-1)(i-2)/2}$
with the notations of \cite{BG12}).  
$R$ and $S$ are determined by the system \cite[Eq.~(1.6)]{BG12}
\begin{equation} \label{eq:2pgenRSeq}
\begin{split}
&S=\sum_{k=1}^\infty t^{k/2} P(k-1,R,S)\\
&R=1+\frac{1}{2} \sum_{k=1}^\infty t^{k/2} P(k,R,S)-\frac{S^2}{2}
\end{split}
\end{equation}
where $P(k,R,S)$ denotes the generating function of three-step paths,
i.e.\ lattice paths in the discrete Cartesian plane consisting of up-steps $(1, 1)$, level-steps $(1, 0)$ 
and down-steps $(1, -1)$, starting at $(0, 0)$ and ending at $(k, 0)$, with a weight $S$ attached to each
level-step and a weight $\sqrt{R}$ attached to each up- or down-step.
The summation over three-step paths yields immediately
 \begin{equation}
S=\sqrt{t}\left(1-2 \sqrt{t}S+t(S^2-4R)\right)^{-1/2}\\
R=1+\frac{1}{2}\left(\frac{S}{\sqrt{t}}-1\right)-\frac{S^2}{2}
\end{equation}
from which we deduce the explicit values
 \begin{equation}
S=\frac{1-\sqrt{1-12 t}}{6\sqrt{t}}\\
R=\frac{1+12t-\sqrt{1-12t}}{18t}.
\end{equation}
It is readily seen that these values coincide with the above expressions $S=\sqrt{t}\,T$ and $R=1+t\, T^2$
for a function $T$ given by
\begin{equation}
T=\frac{1-\sqrt{1-12 t}}{6t}
\end{equation}
which is precisely the solution of $T=1+3t\, T^2$. 
This corroborates, as it should, our results for $R$ and $S$. 

As for $u_i$ and $\tilde{u}_i$, it is known that, in the case of maps with bounded face degrees, the related Hankel determinants 
may be expressed as symplectic Schur functions whose variables are the solutions $x$ of a \emph{characteristic equation}
\cite[Eq.~(1.10)]{BG12} (up to a $x \to 1/x$ symmetry,  this equation admits as many solutions as the maximal allowed face degree minus $2$). As such, $u_i$ and $\tilde{u}_i$ may be expressed in terms of 
determinants of a fixed size, independent of $i$. For unbounded face degrees however, we have no such simplification 
\emph{a priori} and in the present case, we have not been able to derive simple expressions for $u_i$ and $\tilde{u}_i$ directly 
from their expressions via Hankel determinants. Still it is instructive to write the characteristic equation
\begin{equation}
1=\sum_{k=2}^\infty t^{k/2}\sum_{q=0}^{k-2}P(k-2-q,R,S)\left(\sqrt{R}x+S+\frac{\sqrt{R}}{x}\right)^q\ .
\end{equation}
By exchanging the sums, it may be rewritten as 
\begin{equation}
1=\frac{t \left(1-2 \sqrt{t}S+t(S^2-4R)\right)^{-1/2}}{1-\sqrt{t}\left(\sqrt{R}x+S+\frac{\sqrt{R}}{x}\right)}
\end{equation}
which, upon setting $S=\sqrt{t}T$ and $R=1+t\, T^2$, and using $T=1+3t\, T^2$, simplifies
into 
\begin{equation}
 x^2+\frac{1}{x^2}+1=\frac{1}{t\, T^2}\ .
\end{equation}
Note that, up to obvious symmetries $x\to -x$ and $x\to 1/x$, this equation determines 
a unique solution. Moreover, comparing with our bijective results, we are led to the identification $x^2=y$, while
eqs.~\eqref{Rigm} and \eqref{Sigm} show that $u_i= c \lambda^i(1-x^{2i+2})$ for some (undetermined) $c$ and $\lambda$
and $\tilde{u}_i=c  \lambda^i x (1-x^{2i})$ . It is remarkable that $u_i$ and $\tilde{u}_i$ admit such a simple form:
this property still awaits a proper explanation in the continued fraction approach.

\subsubsection{Applications}

As a simple application of the above formulas, we may compute the average number of edges of type $(i-1,i)$, $(i,i)$ or
$(i+1,i)$ in an infinitely large vertex-pointed map, i.e.\ a vertex-pointed map with $n$ edges in the limit $n\to \infty$.
(Note that, in all rigor, our computation incorporates a symmetry factor $1/k$ to those vertex-pointed maps having
a $k$-fold symmetry. These symmetric maps are however negligible in the large $n$ limit.)
The large $n$ asymptotics is easily captured by the singularity of the above generating functions when $t\to 1/12$. 
We have singularities of the form
\begin{equation}
R_i|_\text{sing.} \sim (1-12t)^{3/2}\delta_i \qquad 
S_i^2|_\text{sing.} \sim (1-12t)^{3/2}\eta_i  
\end{equation}
with values of $\delta_i$ and $\eta_i$ easily computed from the exact expressions above for $R_i$ and $S_i$.
By a standard argument, the average numbers of edges of type $(i-1,i)$,  $(i,i)$ and 
$(i+1,i)$ are found respectively to be
\begin{equation}
\begin{split}
&\hskip -1.cm e_{i-1,i}=\frac{3}{2} (\delta_i-\delta_{i-1})=\frac{i (i+3) (2 i+3) \left(5 i^4+30 i^3+67 i^2+66 i+28\right)}{35 (i+1)^2 (i+2)^2} \\
&\hskip -1.cm e_{i,i}=\frac{3}{2} (\eta_i-\eta_{i-1})\\
&\hskip -.5cm=\frac{2 \left(5 i^8+80 i^7+537 i^6+1964 i^5+4251 i^4+5528 i^3+4175 
i^2+1660 i+280\right)}{35 (i+1)^2 (i+2) (i+3)^2} \\
&\hskip -1.cm e_{i+1,i}= e_{i,i+1}=\frac{3}{2} (\delta_{i+1}-\delta_{i})\\
&\hskip -.2cm=\frac{(i+1) (i+4) (2 i+5) \left(5 i^4+50 i^3+187 i^2+310 i+196\right)}
{35 (i+2)^2 (i+3)^2}\\
\end{split}
\end{equation} 
for $i\geq 0$. We have in particular an average number $e_{0,1}=28/9$ (resp. $e_{0,0}=8/9$) of half edges incident to the 
pointed vertex whose complementary half edge is incident to a distinct (resp. the same) vertex.  These two numbers add up to $4$,
as expected since a large map has asymptotically $4$ times more half edges than vertices.
From the singularity 
\begin{equation}
\log(R_i)|_\text{sing.} \sim (1-12t)^{3/2}\theta_i 
\end{equation}
we easily deduce the average number of vertices at distance $i$ from the pointed vertex in 
infinitely large vertex-pointed maps:
\begin{equation}
v_{i}=\frac{3}{2} (\theta_i-\theta_{i-1})=\frac{3}{280} (2 i+3) \left(10 i^2+30 i+9\right)
\end{equation} 
for $i\geq 1$.

\subsection{The two-point function of bipartite maps}
\label{sec:2pbip}
We may now easily play the same game with general \emph{bipartite} maps, which, upon blowing their
edges into dark faces of degree $2$, are nothing but general $2$-constellations. 
\subsubsection{Computation from the bijective approach}
Considering a pointed
rooted bipartite map, its root edge is now necessarily of type $(j-1,j)$ or $(j+1,j)$ for some $j$. 
We use the same notation $R_i\equiv R_i(t)$ to now denote the generating function of pointed rooted 
bipartite maps enumerated with a weight $t$ per edge, whose root edge is of type $(j-1,j)$ for $j\leq i$.
This is also the generating function for vertex-pointed and rooted  $2$-constellations, with a weight $t$ per dark face, 
and a root face of ccw-type $\tau=(j,j-1)$ for $j\leq i$ if we endow the constellation with its geodesic labelling. 
From Proposition~\ref{prop:const1}, the later are in one-to-one correspondence with 
$2$-descending mobiles with a marked black vertex of cw-type $\ubtau=(j,j+1)$, $j\leq i$.

Denoting $T_i=T_i(t)$ the generating function of floating (i.e.\ with arbitrary positive minimal label) $2$-descending 
mobiles planted at a white vertex labelled $i$, enumerated with a weight $t$ per black vertex, 
we deduce immediately the same identification as before
\begin{equation} \label{eq:2pbiprti}
R_i=1+t \,T_i T_{i+1}\qquad i\geq 1\ .
\end{equation}
The generating function $T_i$ is now obtained by solving the equation
\begin{equation}
T_i=\frac{1}{1-t\, (T_{i-1}+T_{i+1})},\qquad i\geq 1
\end{equation}
with initial condition $T_0=0$, and one finds \cite{ONEWALL}
\begin{equation} \label{eq:biptisol}
\begin{split}
&T_i=T \frac{(1-y^i)(1-y^{i+4})}{(1-y^{i+1})(1-y^{i+3})} \\
&\text{where}\ \ T=1+2 t\,T^2\ \ \text{and}\ \ y+\frac{1}{y}=\frac{1}{t\, T^2}\ .
\end{split}
\end{equation}
This leads, after simplification, to
\begin{equation}
\begin{split}
&R_i=R \frac{(1-y^{i+1})(1-y^{i+4})}{(1-y^{i+2})(1-y^{i+3})}\\
&\text{where}\ \ R=1+ t\,T^2\ .
\end{split}
\end{equation}

We may also compute along the same lines as before the generating function $V_i$ for vertex-pointed bipartite maps, 
enumerated with a weight $t$ per edge, with an extra marked vertex at distance $j\leq i$ from the pointed vertex, with $i\geq 1$. 
Following the same chain of arguments as above, we have
\begin{equation}
\begin{split}
&V_i=\sum_{k\geq 1}\frac{\left(t\,(T_{i-1}+T_{i+1})\right)^k}{k}-\sum_{k\geq 1}\frac{\left(t\, T_{i-1}\right)^k}{k}\\
&\ \ =\log\left(\frac{1-t\, T_{i-1}}{1-t\,(T_{i-1}+T_{i+1})} \right)\\
&\ \ =\log\left(T_i(1-t\, T_{i-1}\right)\\
&\ \ =\log\left(1+t\, T_{i}T_{i+1} \right)\\
&\ \ =\log(R_i)\ ,
\end{split}
\end{equation} 
valid for $i\geq 1$.  Note that the relation between $V_i$ and $R_i$ is unchanged when going from general to bipartite maps
and one can argue that it holds for all classes of maps and hypermaps described here, as a consequence of the general BDG bijection.

\subsubsection{Comparison with the continued fraction approach}
Again, part of this result may be re-derived from the approach of \cite{BG12}
by noting that enumerating bipartite maps with weight $t$ per edge amounts to 
enumerating bipartite maps with unbounded even face degrees and with a weight 
$t^k$ per $2k$-valent face. The generating function $R$ is obtained in this
framework via
\begin{equation}
R=1+\sum_{k\geq 1}t^k {2k-1\choose k-1} R^k=\frac{1+\sqrt{1-4 t R}}{2\sqrt{1-4 t R}},
\end{equation}
namely
\begin{equation}
R=\frac{1-\sqrt{1-8t}+4t}{8t}\ .
\end{equation}
This expression is compatible with $R=1+t\,T^2$ for
\begin{equation}
T=\frac{1-\sqrt{1-8t}}{4t}
\end{equation}
which is the solution of $T=1+2 t\,T^2$, as wanted.
For bipartite maps, the generating function $R_i$ is now expected to take the form \cite{GEOD}
\begin{equation} \label{eq:2pbipcontfracform}
R_i=R \frac{u_i u_{i+3}}{u_{i+1}u_{i+2}}\\
\end{equation}
and the characteristic equation reads
\begin{equation}
1=\sum_{k=1}^\infty t^k R^{k-1}\sum_{q=0}^{k-1}{2k-2-2q\choose k-1-q}\left(x+\frac{1}{x}\right)^{2q}\ =\frac{t (1-4 t R)^{-1/2}}{
1-t R \left(x+\frac{1}{x}\right)^2}.
\end{equation}
Setting $R=1+t\, T^2$ and using $T=1+2t\, T^2$, this simplifies
into 
\begin{equation}
 x^2+\frac{1}{x^2}=\frac{1}{t\, T^2}
\end{equation}
which again allows us to identify $x^2=y$ and deduce the simple form 
$u_i=c \lambda^i (1-x^{2i+2})$ for some $c$ and $\lambda$. Getting this 
expression for $u_i$ via the continued fraction approach is still an open question. 

\subsubsection{Applications}
Again, we may compute the average number of edges of type $(i-1,i)$ or
$(i+1,i)$ in an infinitely large vertex-pointed bipartite map, as well as the average number of vertices at distance $i$.
We find
\begin{equation}
\begin{split}
&\hskip -1.cm e_{i-1,i}=\frac{2 i (i+4) \left(10 i^4+80 i^3+233 i^2+292 i+141\right)}{105 (i+1) (i+2) (i+3)} \\
&\hskip -1.cm e_{i+1,i}= e_{i,i+1}=\frac{2 (i+1) (i+5) \left(10 i^4+120 i^3+533 i^2+1038 i+756\right)}{105 (i+2) (i+3) (i+4)}\\
\end{split}
\end{equation} 
for $i\geq 0$ and
\begin{equation}
v_{i}=\frac{4}{315} (i+2) \left(10 i^2+40 i+13\right)
\end{equation} 
for $i\geq 1$. We have in particular an average number $e_{0,1}=3$ of half edges incident to the 
pointed vertex, as expected since a large bipartite map has asymptotically $3$ times more half edges than vertices
(this is easily seen from the ``trivial'' bijection between bipartite maps and Eulerian triangulations).

\subsubsection{The two-point functions of general hypermaps}\label{sec:genhyper}

The recourse to $2$-descending mobiles used above for the computation of the two-point function of bipartite maps
turns out to be also helpful to compute the two-point function of general hypermaps.
Indeed, upon using Remark~\ref{rmk:2dmhyp}, it can be shown that the generating function 
$\mathcal{R}_i=\mathcal{R}_i(t)$ of vertex-pointed general
hypermaps with a marked edge of type $(j-1,j)$ for $j\leq i$, enumerated with a weight $t$
per edge of the hypermap, is identical to that of triples of consecutive white labelled
vertices of labels $(j,j+1,j+2)$ with $j\leq i$ in $2$-descending mobiles, with a weight $t$ per black vertex. 
A proof of this statement is given just below. The two-point function $\mathcal{R}_i$ therefore reads
\begin{equation}
\mathcal{R}_i= 1+ t^2 T_iT_{i+1}T_{i+2}\ ,
\label{eq:calR}
\end{equation} 
$i\geq 1$, with $T_i$ as in \eqref{eq:biptisol} (again we incorporate in $\mathcal{R}_i$ a conventional term $1$).  Using
the explicit form of $T_i$, we immediately deduce the factorized form
\begin{equation}
\begin{split}
&\mathcal{R}_i=\mathcal{R} \frac{(1-y^{i+2})(1-y^{i+4})}{(1-y^{i+3})^2}\\
&\text{where}\ \ \mathcal{R}=1+ t^2\,T^3\ .
\end{split}
\end{equation}
Equation \eqref{eq:calR} is a consequence of the following Claim:

\begin{claim}\label{claim:hypertomob}
In the bijection of Remark~\ref{rmk:2dmhyp} between vertex-pointed hypermaps (endowed with their 
geodesic labelling) and 2-descending mobiles, each edge $(i-1,i)$
of the hypermap corresponds to a triple of consecutive white labelled
vertices of labels $(i,i+1,i+2)$ in counterclockwise order around
the mobile. And the vertex of label $i$ of the edge identifies
to the vertex of label $i$ of the triple.
\end{claim} 
\begin{proof}
Recall that the bijection can be seen as the composition of $3$ correspondences:
(1) between vertex-pointed hypermaps and vertex-pointed bipartite maps (Prop.~\ref{prop:bij_bip_hyp}), (2) between
vertex-pointed bipartite maps (viewed as $2$-constellations) and vertex-pointed stretched quadrangulations,
(3) between vertex-pointed stretched quadrangulations and 2-descending
mobiles (the composition of (2) and (3) corresponding to Prop.~\ref{prop:const1}). 
Hence we have to study how an edge $(i-1,i)$ of the hypermap
is transported through each of the 3 steps. Given the local rules
of $\Phi^-$ (see Figure~\ref{fig:rules} right part), it is clear that 
each edge $(i-1,i)$ of a vertex-pointed hypermap corresponds 
to a triple of consecutive vertices of labels $(i-1,i,i+1)$
in clockwise order around a face of the associated vertex-pointed bipartite map. To look at the parameter-correspondence in steps (2) and (3)
we find it simpler to take the point of view of stretched 
quadrangulations. Let $Q$ be a vertex-pointed stretched quadrangulation,
$B$ the corresponding vertex-pointed bipartite map. Each non local max
vertex $v$ of label $i$ in $Q$ identifies to a vertex of label $i$ in $B$. 
By the local rules of $\Phi^-$, a corner $c$ of $Q$ at $v$ yields an edge $e$ 
of $B$
incident to $v$ iff the neighbour (in $Q$) of $v$  
on the left side of $c$ has label $i+1$,
and in that case
the other extremity of $e$ has same label as the neighbour (in $Q$) of $v$ 
 on the right side of $c$. It easily follows that each clockwise-consecutive triple $(i-1,i,i+1)$ in $B$ corresponds in $Q$ to a vertex $v$ of label $i$ together with a triple 
of (clockwise) consecutive neighbours of $v$ of labels $(i+1,i+1,i-1)$. 
Regarding $Q$, 
define an $(i-1)$-sector of $v$ as a corner at $v$ upon removing
the edges $(i,i+1)$ around $v$; note that clockwise-consecutive triples of neighbours of labels $(i+1,i+1,i-1)$ around $v$ are in 1-to-1 correspondence with $(i-1)$-sectors
(around $v$) of multiplicity strictly larger than $1$ (the multiplicity being the number of removed $(i,i+1)$ edges). 
Now we can discuss step (3) of the bijection. Let $T$ be the 2-descending mobile associated to $Q$. 
By the local rules of $\Phi$, each $(i-1)$-sector $s$ at $v$ yields
exactly one edge of $T$ (in the first corner in clockwise order around the sector). In addition, as shown in Figure~\ref{fig:cases_corresp}, 
this edge leads to a white vertex $w$ of label $i-1$ (resp. $i+1$) if $s$ has multiplicity $0$ (resp.\ $>0$),
and in the second case the next vertex after $w$ (in counterclockwise order around the mobile) has label $i$ (resp. $i+2$) if $s$ has multiplicity $1$ (resp.\ $>1$). This concludes the proof.     
\end{proof}

\begin{figure}
\begin{center}
\includegraphics[width=12cm]{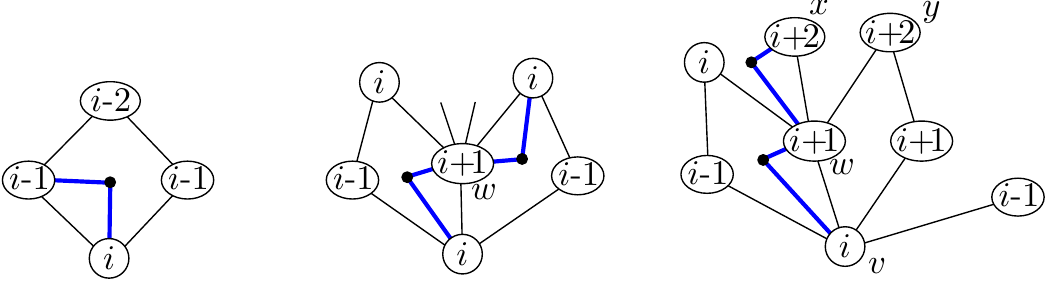}
\end{center}
\caption{The 3 cases for an $(i-1)$-sector at a vertex of label $i$ in $Q$: multiplicity $0$, $1$ and $\geq 2$, respectively. The second case can take the degenerated form where the faces on each side of the edge $(i,i+1)$ are the same, in which
case $w$ is a leaf of the 2-descending mobile. 
In the third case, the successor of $w$ of label $i+2$ around the mobile
 is generically defined
as the unique neighbour $x$ of $w$ such that the neighbour preceding $x$ (in clockwise order around $w$) has label $i$ and all neighbours of $w$ between $x$
and $y$ have label $i+2$.}
\label{fig:cases_corresp}
\end{figure}
 
\subsection{The two-point function of $3$-hypermaps}
\label{sec:2p3h}
Consider now a vertex-pointed and rooted $3$-hypermap (recall that the root edge is
oriented so as to have the dark incident face -- the root face -- on its right). We wish to enumerate 
such $3$-hypermaps endowed with their geodesic labelling and with their root face of
ccw-type $(j_1, j_2,j_3)$. By convention, when giving the type of the root face, we shall always 
start the cyclic sequence from the endpoint of the root edge (in particular, the root edge has its origin 
and endpoints at respective distances $j_2$ and $j_1$ from the pointed vertex). Note that,  
the only possible ccw-types for the root face are of the form $(j,j-1,j-2)$, $(j,j-1,j-1)$, $(j,j-1,j)$, 
$(j,j-1,j+1)$, $(j,j,j-1)$, $(j,j,j)$, $(j,j,j+1)$, $(j,j+1,j)$, $(j,j+1,j+1)$ and $(j,j+2,j+1)$ for some $j$. 
We call $R_{i_1,i_2,i_3}=R_{i_1,i_2,i_3}(t)$ the generating function 
for vertex-pointed and rooted $3$-hypermaps, enumerated with a weight $t$ per dark face,
whose root face is of ccw-type $\tau=(i_1-m,i_2-m,i_3-m)$ for some $m\geq 0$  (clearly $m \leq \min(i_1,i_2,i_3)$
by definition and the allowed values of $(i_1,i_2,i_3)$ have the same form as the ccw-types listed above).
Note that, in a $3$-hypermap, the number of edges is three times the number of dark faces, so 
our counting amounts to attaching a weight $t^{1/3}$ to each edge. 

From Theorem~\ref{theo:hyp_mobiles}, these $3$-hypermaps are in one-to-one correspondence with $3$-mobiles
 with a marked black vertex of cw-type $\ubtau$, one of its incident half-edges being distinguished.
Note that the labels encountered in the sequence $\ubtau$ are precisely $i_1-m+1$, $i_2-m+1$ and $i_3-m+1$ 
(with a prescribed order of apearence). Hence, gathering all cases for all $m\geq 0$ and shifting 
the labels by $m$, we end up with floating $3$-mobiles with a marked black vertex adjacent to white 
vertices with labels  $i_1+1$, $i_2+1$ and $i_3+1$ in a prescribed order.
This allows to write, for the allowed values of $(i_1,i_2,i_3)$
\begin{equation}
R_{i_1,i_2,i_3}=t\, T_{i_1+1}T_{i_2+1}T_{i_3+1}
\end{equation}
where $T_i$ is the generating function of floating $3$-mobiles planted at a white vertex with label $i$.
This later generating function is easily shown to satisfy
\begin{equation}
\ 
\hskip -.5cm T_i=\frac{1}{1-t\,(T_{i-2}T_{i-1}\!+\!T_{i-1}^2\!+\!2 T_{i-1}T_i\!+\!T_i^2\!+\!T_{i-1}T_{i+1}
\!+\!2 T_iT_{i+1}\!+\!T_{i+1}^2\!+\!T_{i+1}T_{i+2})}
\end{equation} 
for $i\geq 1$ with initial conditions $T_0=0$ and $T_{-1}T_0=0$. This equation expresses that a floating $3$-mobiles 
planted at a white vertex 
labelled $i$ may be viewed as a sequence of pairs of planted floating $3$-mobiles -- attached to the root vertex via trivalent black vertices -- with appropriate root labels. As explained in \cite[Section 6.1]{GEOD} and \cite{BG12}, we have the explicit formula
\begin{equation}
\begin{split}
&T_i=T \frac{v_i(y_1,y_2)\, v_{i+3}(y_1,y_2)}{v_{i+1}(y_1,y_2)\, v_{i+2}(y_1,y_2)} \\
&\text{where}\ \ T=1+10 t \,T^3\ .
\end{split}
\end{equation}
Here $y_1$ and $y_2$ are the two solutions (with modulus less than $1$) of
\begin{equation}
y^2+6 y+6+\frac{6}{y}+\frac{1}{y^2}=\frac{1}{t\, T^3}
\end{equation}
(note that, in particular $y_1 $ and $y_2$ are related by $y_1+y_1^{-1}+y_2+y_2^{-1}+6=0$)
and $v_i(y_1,y_2)$ denotes 
\begin{equation}
v_i(y_1,y_2)=1-\frac{1-y_1 y_2}{y_1-y_2} \, y_1^{i+1}-\frac{1-y_1 y_2} {y_2-y_1}\, y_2^{i+1}
- y_1^{i+1} y_2^{i+1}\ . \\ 
\end{equation}
Consider now a vertex-pointed and rooted $3$-hypermap and concentrate on the type of
its root edge, defined as the pair $(k,j)$ of the respective distances of its origin end
endpoint from the pointed vertex. In particular, we denote 
as in previous sections by $R_i\equiv R_i(t)$, $i\geq 1$ the generating function for vertex-pointed 
and rooted $3$-hypermaps with a root edge is of type $(j-1,j)$ for $j\leq i$. 
Listing the possible corresponding ccw-type of the root face, namely $(j,j-1,j-2)$, $(j,j-1,j-1)$, $(j,j-1,j)$ or $(j,j-1,j+1)$,
we deduce
\begin{equation}
\begin{split}
R_i& =1+R_{i+1,i,i-1}+R_{i+1,i,i}+R_{i+1,i,i+1}+R_{i+1,i,i+2}\\
&=1+t\, T_{i}T_{i+1}(T_{i-1}+T_{i}+T_{i+1}+T_{i+2})
\end{split}
\end{equation}
with, as before, a conventional term $1$. After some tedious calculations, we find, using the above mentioned relation between
$y_1$ and $y_2$, the remarkable simplification
\begin{equation}
\begin{split}
&R_i=R \frac{v_{i+1}(y_1,y_2)\, v_{i+3}(y_1,y_2)}{(v_{i+2}(y_1,y_2))^2} \\
&\text{where}\ \ R=1+4 t\,  T^3\ .
\end{split}
\end{equation}
We may finally compute the generating function $V_i$ for vertex-pointed $3$-hypermaps, enumerated with a weight $t$ per 
dark face, with an extra marked vertex at distance $j\leq i$ from the pointed vertex, with $i\geq 1$ (again with
their symmetry factor). From Proposition~\ref{prop:const1}, $V_i$ is the generating function for floating $3$-mobiles
with a marked non-right-local-max vertex labelled $i$ and we may write  
\begin{equation}
\begin{split}
&\hskip -1.2cm V_i=\sum_{k\geq 1}\frac{\left(t\,(T_{i-2}T_{i-1}\!+\!T_{i-1}^2\!+\!2 T_{i-1}T_i\!+\!T_i^2\!+\!T_{i-1}T_{i+1}
\!+\!2 T_iT_{i+1}\!+\!T_{i+1}^2\!+\!T_{i+1}T_{i+2})\right)^k}{k}\\
&-\sum_{k\geq 1}\frac{\left(t\,(T_{i-2}T_{i-1}\!+\!T_{i-1}^2\!+\!2 T_{i-1}T_i\!+\!T_i^2\!+\!T_iT_{i+1})\right)^k}{k}\\
&\hskip -1.2cm \ \ =\log\left(T_i(1-t\, (T_{i-2}T_{i-1}\!+\!T_{i-1}^2\!+\!2 T_{i-1}T_i\!+\!T_i^2\!+\!T_iT_{i+1})\right)\\
&\hskip -1.2cm \ \ =\log\left(1+t\, T_{i}T_{i+1}(T_{i-1}+T_{i}+T_{i+1}+T_{i+2} \right)\\
&\hskip -1.2cm \ \ =\log(R_i)
\end{split}
\end{equation} 
valid for $i\geq 1$ (the subtracted term removes configurations were
the root vertex would be a right-local max). We thus recover the same relation $V_i=\log(R_i)$ as in previous sections,
as expected. 

\subsection{The two-point function of $3$-constellations}
\label{sec:2p3c}
Consider finally vertex-pointed and rooted $3$-constellations endowed with their geodesic labelling,
whose root face is of ccw-type $(i_1-m, i_2-m,i_3-m)$ for some $m\geq 0$ and denote again 
by $R_{i_1,i_2,i_3}=R_{i_1,i_2,i_3}(t)$ their generating function. The possible values of
$(i_1,i_2,i_3)$ are now restricted to $(i,i-1,i-2)$,  
$(i,i-1,i+1)$ and $(i,i+2,i+1)$ for some $i$. 

\subsubsection{Computation from the bijective approach}
From Proposition~\ref{prop:const1}, vertex-pointed and rooted $3$-constellations with root face of ccw-type $\tau$ 
are in one-to-one correspondence with $3$-descending mobiles
with a marked black vertex of cw-type $\ubtau$, one of its incident half-edges being distinguished.
By the same argument as before, we find
\begin{equation}
R_{i_1,i_2,i_3}=t\, T_{i_1+1}T_{i_2+1}T_{i_3+1}
\end{equation}
for the allowed values of $(i_1,i_2,i_3)$.
Here $T_i$ denotes the generating function for floating $3$-descending mobiles planted at a white vertex labelled $i$. It
satisfies
\begin{equation}
T_i=\frac{1}{1-t\,(T_{i-2}T_{i-1}+T_{i-1}T_{i+1}+T_{i+1}T_{i+2})}
\end{equation} 
for $i\geq 1$ with initial conditions $T_0=0$ and $T_{-1}T_0=0$.
The solution of this equation was found (in the context of Eulerian quadrangulations viewed as a particular case
of $4$-constellations -- see \cite{GEOD,PDFRaman}) to be
\begin{equation}
\begin{split}
&T_i=T \frac{v_i(y_1,y_2)\, v_{i+5}(y_1,y_2)}{v_{i+1}(y_1,y_2)\, v_{i+4}(y_1,y_2)} \\
&\text{where}\ \ T=1+3 t \,T^3\ .
\end{split}
\end{equation}
Here $y_1$ and $y_2$ are the two solutions (with modulus less than $1$) of
\begin{equation}
y^2+2 y+\frac{2}{y}+\frac{1}{y^2}=\frac{1}{t\, T^3}
\end{equation}
(note that, in particular $y_1 $ and $y_2$ are related by $y_1+y_1^{-1}+y_2+y_2^{-1}+2=0$)
and $v_i(y_1,y_2)$ now denotes 
\begin{equation}
\begin{split}
&v_i(y_1,y_2)=1-\frac{p_1-y_1^4 p_2}{p_1-p_2} \, y_1^{i}-\frac{p_2-y_2^4 p_1} {p_2-p_1}\, y_2^{i}
+\frac{y_2^4p_1-y_1^4p_2}{p_1-p_2} y_1^{i} y_2^{i}\\ 
&\text{where}\ \ p_1=y_1+y_1^2+y_1^3\ , \quad p_2=y_2+y_2^2+y_2^3 \ .\\
\end{split}
\end{equation}
We may as before concentrate on the type of the root edge only and look for instance at the generating function
$R_i\equiv R_i(t)$, $i\geq 1$ for vertex-pointed 
and rooted $3$-constellations with a root edge of type $(j-1,j)$ for $j\leq i$. 
By listing the possible ccw-type of the root face, namely 
$(j,j-1,j-2)$ and $(j,j-1,j+1)$, we obtain
\begin{equation}
\begin{split}
R_i& =1+R_{i+1,i,i-1}+R_{i+1,i,i+2}\\
&=1+t\, T_{i}T_{i+1}(T_{i-1}+T_{i+2})
\end{split}
\end{equation}
with, as before, a conventional term $1$. 
We find the remarkable simplification
\begin{equation}
\begin{split}
&R_i=R \frac{v_{i+1}(y_1,y_2)\, v_{i+5}(y_1,y_2)}{v_{i+2}(y_1,y_2)v_{i+4}(y_1,y_2)} \\
&\text{where}\ \ R=1+2 t\,  T^3\ .
\end{split}
\end{equation}
We may as before evaluate the generating function $V_i$ for vertex-pointed $3$-constella\-tions, enumerated with a weight $t$ per 
dark face, with an extra marked vertex at distance $j\leq i$ from the pointed vertex, with $i\geq 1$ (with
symmetry factors). Repeating the arguments of previous Sections, we again have the relation $V_i=\log(R_i)$. 

\subsubsection{Comment on the form of the two-point function}
\label{sec:comm2pconst}
Let us now comment on the respective forms of $T_i$ and $R_i$: first we note that the indices involved in the 
respective bi-ratios of $v$'s match what we expect for constellations. As discussed in \cite{GEOD,PDFRaman},
we expect the two-point function 
for pointed rooted $p$-constellations to display a bi-ratio of the form $(U_i / U_{i+1})/(U_{i+p}/U_{i+p+1})$. 
The observed form of $T_i$ is thus typical of $4$-constellations (with $U_i \leftrightarrow v_i$) in agreement with the fact that, 
from the BDG bijection, $T_i$ may be interpreted as a two-point function for $4$-regular constellations (i.e.\ Eulerian quadrangulations).
As for the form of $R_i$, it is typical of $3$-constellations (with $U_i \leftrightarrow v_{i+1}$) as it should.
A second remark concerns the precise value of the function $v_i(y_1,y_2)$. Its observed form is characteristic of $4$-constellations
while, for a $3$-constellation, one would have expected instead \cite{GEOD,PDFRaman}
\begin{equation}
\begin{split}
&v_i(y_1,y_2)=1-\frac{p'_1-y_1^3 p'_2}{p'_1-p'_2} \, y_1^{i}-\frac{p'_2-y_2^3 p'_1} {p'_2-p'_1}\, y_2^{i}
+\frac{y_2^3p'_1-y_1^3p'_2}{p'_1-p'_2} y_1^{i} y_2^{i}\\ 
&\text{where}\ \ p'_1=y_1+y_1^2\ , \quad p'_2=y_2+y_2^2 \\
\end{split}
\end{equation}
Remarkably enough, the two expression do coincide whenever $y_1+y_1^{-1}+y_2+y_2^{-1}+2=0$,
which is precisely the above mentioned condition satisfied by $y_1$ and $y_2$ in the solution for $T_i$. 
In other words,  the observed form for $v_i(y_1,y_2)$  matches both that expected for $4$-constellations and 
that expected for $3$-constellations, a non-trivial property. 
Note that the relation between $y_1$ and $y_2$ (and therefore the coincidence of the two expressions for $v_i$) holds only when 
dealing with a very specific family of $4$-constellations, namely the $4$-regular ones.

This property generalizes as follows: the expected form for the two-point function $R_i$ for a $p$-constellation is \cite{GEOD,PDFRaman},
\begin{equation}
R_i=R \frac{v_i v_{i+p+1}}{v_{i+1}v_{i+p}}
\end{equation}
where $v_i=v_i(y_1,y_2, \cdots,y_m)$ takes the form
\begin{equation}
v_i(y_1,y_2,\cdots,y_m)=\sum_{I\subset \{1,2,\cdots,m\}} \prod_{k\in I}\lambda_k y_k^i \prod_{k,k'\in I} 
\frac{(p^{(p)}_k-p^{(p)}_{k'})(q^{(p)}_k-q^{(p)}_{k'})}{(p^{(p)}_k-q^{(p)}_{k'})(q^{(p)}_k-p^{(p)}_{k'})}\ .
\end{equation}
Here the sum is over all subsets $I$ of $\{1,2,\cdots,m\}$. 
The quantities $y_i$ are the solutions (with modulus less than one) of a characteristic equation depending on
the problem at hand (and their number $m$ depends on the problem too) and the $\lambda_k$ are fixed by 
demanding $v_0=v_{_1}=v_{-2}=\cdots=v_{-m+1}=0$.
Finally, $p^{(p)}_k$ and $q^{(p)}_k$ are defined as
\begin{equation}
p^{(p)}_k=y_k+y_k^2+\cdots +y_k^{p-1}\ , \qquad q^{(p)}_k=y_k^{-1}+y_k^{-2}+\cdots +y_k^{-p+1}\ .
\end{equation}
The reader will easily check that all the specific expressions given above for $2$-, $3$- and $4$-constellations
are indeed of this form. Now, from our bijections, the two-point function $R_i$ in the case
of $p$-constellations enumerated with a weight $t$ per dark face is easily expressed in terms of the generating function $T_i$ for
descending $p$-mobiles with a weight $t$ per black vertex, which is itself the two-point function for a particular instance of $(p+1)$-constellations,
namely $(p+1)$-regular constellations (i.e.\ with all their dark and light faces of degree $(p+1)$), enumerated
with a weight $t$ per dark face. As such, $T_i$ 
takes the form expected for a $(p+1)$-constellation (with moreover $m=p-1$ in this case)
\begin{equation}
T_i=T \frac{v_i v_{i+p+2}}{v_{i+1}v_{i+p+1}}
\end{equation}
where $v_i=v_i(y_1,y_2, \cdots,y_{p-1})$ now reads
\begin{equation}
\ 
\hskip -1.cm v_i(y_1,y_2,\cdots,y_{p-1})=\!\! \sum_{I\subset \{1,2,\cdots,p-1\}} \prod_{k\in I}\lambda_k y_k^i \prod_{k,k'\in I} 
\frac{(p^{(p+1)}_k-p^{(p+1)}_{k'})(q^{(p+1)}_k-q^{(p+1)}_{k'})}{(p^{(p+1)}_k-q^{(p+1)}_{k'})(q^{(p+1)}_k-p^{(p+1)}_{k'})}\ .
\end{equation}
The reader may again check that the specific expressions given above for $3$- and $4$-regular constellations
are indeed of this form. It is tempting to conjecture that the $v_i$ appearing in $R_i$ is the same as that appearing in $T_i$
(we have seen above that this holds for $p=2$ and $p=3$) and that the precise relation between $R_i$ and $T_i$ 
is responsible for the change of indices in the involved bi-ratios.
As a check of consistency for this conjecture to hold for general $p$, we can verify that $v_i(y_1,y_2, \cdots,y_{p-1})$,
as given just above, may be rewritten in the form expected for a $p$-constellation, which requires
\begin{equation}
\frac{(p^{(p+1)}_k-p^{(p+1)}_{k'})(q^{(p+1)}_k-q^{(p+1)}_{k'})}{(p^{(p+1)}_k-q^{(p+1)}_{k'})(q^{(p+1)}_k-p^{(p+1)}_{k'})}
=\frac{(p^{(p)}_k-p^{(p)}_{k'})(q^{(p)}_k-q^{(p)}_{k'})}{(p^{(p)}_k-q^{(p)}_{k'})(q^{(p)}_k-p^{(p)}_{k'})}
\label{cpq}
\end{equation}
for all pairs $\{y_k,y_{k'}\}$ of distinct solutions of the characteristic equation for $(p+1)$-regular constellations. The later reads
explicitly 
\begin{equation}
\begin{split}
&H(y)\equiv \sum_{k=1}^{p-1} (p-k) \left(y^k+\frac{1}{y^k}\right)=\frac{1}{t\, T^{p}}\\
&\text{where}\ \ T=1+p\,t\,  T^p\ .
\end{split}
\end{equation}
To prove \eqref{cpq}, we note, using $y_k^p q_k^{(p)}=p_k^{(p)}$, that  
\begin{equation}
\frac{(p^{(p)}_k-p^{(p)}_{k'})(q^{(p)}_k-q^{(p)}_{k'})}{(p^{(p)}_k-q^{(p)}_{k'})(q^{(p)}_k-p^{(p)}_{k'})}
=\frac{a^{(p)}_{kk'}+a^{(p)}_{k'k}-1}{a^{(p)}_{kk'}a^{(p)}_{k'k}},\qquad a^{(p)}_{kk'}
=\frac{y_{k'}^pp_k^{(p)}-p_{k'}^{(p)}}{p_k^{(p)}-p_{k'}^{(p)}}
\end{equation}
so \eqref{cpq} is satisfied if $a^{(p)}_{kk'}= a^{(p+1)}_{kk'}$, or equivalently
\begin{equation}
\begin{split}
&0=(y_{k'}^pp_k^{(p)}-p_{k'}^{(p)})(p_k^{(p+1)}-p_{k'}^{(p+1)})-(y_{k'}^{p+1}p_k^{(p+1)}-p_{k'}^{(p+1)})(p_k^{(p)}-p_{k'}^{(p)})\\
& \ \ = (y_k y_{k'})^p (1-y_{k'})(H(y_k)-H(y_{k'})),
\end{split}
\end{equation}
which holds precisely whenever $y_k$ and $y_{k'}$ are two solutions of the characteristic equation, since $H(y_k)=H(y_{k'})$
in this case.
That the passage from $T_i$ to $R_i$ induces the wanted change of indices in the involved bi-ratios
of $v_i$ is far from obvious and we have not found any simple argument that would explain this property for general $p$.

\section{Two-point functions depending on two parameters}
\label{sec:2p2par}   

So far, we have not exploited the full power of the bijection of
Theorem~\ref{theo:hyp_mobiles}, and in particular the property that it
transforms the faces of a (hyper)map into right local max of the
corresponding mobile. In this section, we will use this property in
order to derive the two-point function of general planar maps and
bipartite planar maps counted according to \emph{two size parameters}:
the number of edges and the number of faces. (Note that the number of
vertices is then fixed by Euler's relation.)

The case of general planar maps was actually first treated in
\cite[Section 5]{AmBudd}, but we will recall its derivation for
completeness.

\subsection{The two-point function of general maps with edge and face weights}
\label{sec:2pgen2par}

Our purpose is to obtain a generalization of the two-point function of
general maps derived in Section~\ref{sec:2pgen} depending on two
formal variables: the previous weight $t$ per edge, and an extra
weight $z$ per face. We still denote by $R_i \equiv R_i(t,z)$, $i \geq
1$ (resp.\ $S_i^2 \equiv S_i(t,z)^2$, $i \geq 0$) the generating
function of pointed rooted maps whose root edge is of type $(j-1,j)$
(resp.\ $(j,j)$) for $j \leq i$.

By Theorem~\ref{theo:hyp_mobiles}, these generating functions may
alternatively be understood as counting some floating $2$-mobiles, and
the variable $z$ now plays the role of a weight per right local max.
Note that, upon removing the bivalent black vertices, a right local max is
nothing but a vertex having no neighbor with strictly larger label. We still denote by
$T_i \equiv T_i(t,z)$ the generating function of floating $2$-mobiles
planted at a white vertex labelled $i$, with this extra variable. In
order to extend \eqref{eq:gentrec} to the case $z \neq 1$, we shall
introduce another generating function $U_i \equiv U_i(t,z)$ counting
the same objects, but where the root vertex does not receive the
weight $z$ if it is a local max. We then have the following equations
\begin{equation}
  \label{eq:gentuieq}
    T_i = z + t (T_i U_{i-1} + T_i^2 + U_i T_{i+1}), \qquad
    U_i = 1 + t (U_i U_{i-1} + U_i T_i + U_i T_{i+1}),
\end{equation}
with initial data $T_0=U_0=0$, which are obtained by a straightforward
recursive decomposition. Clearly, these equations admit a unique
solution in the set of sequences of formal power series in $t$ and
$z$.

Ambj\o rn and Budd derived these equations in a slightly different
form and, remarkably, found an exact expression for their solution. It
takes the form \cite[Eq.~(26)]{AmBudd}
\begin{equation}
  \label{eq:gentuisol}
    T_i = T \frac{(1-y^i)(1-\alpha^2 y^{i+3})}{
      (1-\alpha y^{i+1})(1-\alpha y^{i+2})}, \qquad
    U_i = U \frac{(1-y^i)(1-\alpha y^{i+3})}{
      (1-y^{i+1})(1-\alpha y^{i+2})},
\end{equation}
where $T$, $U$, $y$ and $\alpha$ may be determined by the condition that this
ansatz
satisfies the equations. In practice, substituting
\eqref{eq:gentuisol} into \eqref{eq:gentuieq} yields a system of two
equations of the form
\begin{equation}
  \label{eq:polform}
  P_s(t,z,T,U,y,\alpha,y^i) = 0, \qquad
  s=1,2,
\end{equation}
where $P_1,P_2$ are seven-variable polynomials that are
independent of $i$. Thus, expanding $P_1,P_2$ with respect to the last
variable, all the coefficients should vanish identically, which yields
a system of algebraic equations relating
$t,z,T,U,y,\alpha$. Miraculously, this system defines a
two-dimensional variety, which allows to determine the ``unknowns''
$T,U,y,\alpha$ as algebraic power series in the variables $t$ and $z$.
Of course, we find that $T$ and $U$ are nothing but the constant
solutions of \eqref{eq:gentuieq}, namely they are specified by the
tree equations
\begin{equation}
  \label{eq:gentueq}
  T = z + t (T^2 + 2 T U), \qquad U = 1 + t (2 T U + U^2),
\end{equation}
while we find that $y$ and $\alpha$ are obtained by inverting the relations
\begin{equation}
  \label{eq:genyalpeq}
  \begin{split}
    t &= \frac{y (1 - \alpha y)^3 (1 - \alpha y^3)}{ (1+y+\alpha y-6
      \alpha y^2+\alpha y^3 + \alpha^2 y^3 + \alpha^2 y^4)^2},\\
    z &= \frac{\alpha (1 - y)^3 (1 - \alpha^2 y^3)}{
      (1 - \alpha y)^3 (1 - \alpha y^3)}. \\
  \end{split}
\end{equation}
The first few terms of $y$ and $\alpha$ read
\begin{equation}
  \begin{split}
    y &= t+(2+5z)t^2+(5+31z+23z^2)t^3+(14+153z+275z^2+102z^3)t^4\cdots,\\
    \alpha &= z+3z(1-z)t+3z(1-z)(4+z)t^2+z(1-z)(49+51z+4z^2)t^3 + \cdots, \\
  \end{split}
\end{equation}
and it can be shown that $y$ and $\alpha$ are indeed series in $t$
whose coefficients are polynomials in $z$ with integer
coefficients. Furthermore, it appears that all coefficients of $y$ and
$\alpha y$ are nonnegative, which suggests a possible combinatorial
interpretation. For bookkeeping purposes, let us mention the nice
relations
\begin{equation}
  \frac{U}{T} = \frac{(1-\alpha y)^2}{\alpha (1-y)^2}, \qquad
  t T^2 = \frac{\alpha^2 y (1-y)^4}{(1 - \alpha y)^3 (1 - \alpha y^3)}.
\end{equation}

Recalling now the actual mobile interpretation of the generating
functions $R_i$ and $S_i$ of Section~\ref{sec:2pgen}, we easily obtain
the expressions
\begin{equation}
  \begin{split}
    R_i&= 1 + t\, U_i T_{i+1}, \qquad i \geq 1,\\
    S_i&= \sqrt{t}\, T_i, \qquad i \geq 0,\\
  \end{split}
\end{equation}
which extend \eqref{eq:genrsti} in the case $z \neq 1$. Interestingly,
we find that $R_i$ reads
\begin{equation}
  R_i = R \frac{(1-\alpha y^{i+1}) (1-\alpha y^{i+3})}{(1-\alpha y^{i+2})^2}, \qquad
  R = 1+t\, UT=\frac{(1-\alpha y^2)^2}{(1-\alpha y)(1-\alpha y^3)},
\end{equation}
which naturally matches the form \eqref{eq:2pgencontfracform} expected from
the continued fraction approach, since our problem now amounts to
considering maps with a weight $g_k=z t^{k/2}$ per face of degree
$k$. As for $S_i$, it might also be written in the form
\eqref{eq:2pgencontfracform}, albeit in a non-unique manner. We might also
repeat the exercise of Section~\ref{sec:2pgencf} and check that $R$
and $S=\sqrt{t}T$ satisfy the relations
\begin{equation} \label{eq:2pgen2parRSeq}
  S=z \sum_{k=1}^\infty t^{k/2} P(k-1,R,S), \qquad
  R=1+\frac{z}{2} \sum_{k=1}^\infty t^{k/2} P(k,R,S)-\frac{S^2}{2}
\end{equation}
extending \eqref{eq:2pgenRSeq}. Doing a full derivation of the
two-point functions from the continued fraction approach is a
challenging open question.

\subsection{The two-point function of bipartite maps with edge and face weights}
\label{sec:2pbip2par}

We may attempt the same approach in the case of bipartite maps, to
obtain a two-variable generalization of the two-point function derived
in Section~\ref{sec:2pbip}. Let $R_i \equiv R_i(t,z)$, $i \geq 1$, be
the generating function of pointed rooted bipartite maps whose root
edge is of type $(j-1,j)$, $j \leq i$, counted with a weight $t$ per
edge and $z$ per face. By Proposition~\ref{prop:const1}, such maps are
in bijection with some floating $2$-descending mobiles, with $z$ a
weight per (right) local max. Denote again by $T_i \equiv T_i(t,z)$
(resp.\ $U_i \equiv U_i(t,z)$) the generating functions of floating
$2$-descending mobiles planted at a white vertex labelled $i$, counted
with a weight $t$ per black vertex and $z$ per right local max (resp.\
right local max distinct from the root vertex). We now find the
recursive decomposition equations
\begin{equation}
  \label{eq:biptuieq}
    T_i = z + t (T_i U_{i-1} + U_i T_{i+1}), \qquad
    U_i = 1 + t (U_i U_{i-1} + U_i T_{i+1}),
\end{equation}
with initial data $T_0=U_0=0$.

Inspired by the Ambj\o rn-Budd solution \eqref{eq:gentuisol} for
general maps and by the form \eqref{eq:biptisol} for $T_i=U_i$ at
$z=1$, we make the ansatz
\begin{equation}
  \label{eq:biptuisol}
      T_i = T \frac{(1-y^i)(1-\alpha^2 y^{i+4})}{
      (1-\alpha y^{i+1})(1-\alpha y^{i+3})}, \qquad
    U_i = U \frac{(1-y^i)(1-\alpha y^{i+4})}{
      (1-y^{i+1})(1-\alpha y^{i+3})}.
\end{equation}
Repeating the same strategy as in the previous section we find,
miraculously again, that there exists $T$, $U$, $y$ and $\alpha$ which
are algebraic power series in $t$ and $z$ such that the ansatz
satisfies the equations. As expected $T$ and $U$ are specified by
\begin{equation}
  \label{eq:biptueq}
  T = z + 2 t T U, \qquad U = 1 + t U (T+U),
\end{equation}
while $y$ and $\alpha$ are obtained by inverting
\begin{equation}
  \label{eq:bipyalpeq}
  t = \frac{y (1-\alpha y)^2 (1-\alpha y^4)}{(1+y)^2 (1-\alpha y^2)^3}, \qquad
  z = \frac{\alpha (1-y)^2 (1-y^2) (1+\alpha y^2)}{
    (1-\alpha y)^2 (1-\alpha y^4)}.
\end{equation}
We note that $T$ and $U$ are parametrized in terms of $y$ and $\alpha$ by
\begin{equation}
  \label{eq:biptuparam}
  T = \frac{\alpha (1-y^2)^2 (1-\alpha y^2)}{(1-\alpha y)^2 (1-\alpha y^4)},
  \qquad U = \frac{(1+y) (1-\alpha y^2)^2}{(1-\alpha y) (1-\alpha y^4)} 
\end{equation}
and that the first few terms of $y$ and $\alpha$ read
\begin{equation}
  \begin{split}
    y &= t+2(1+z)t^2+(5+13z+3z^2)t^3+(14+66z+40z^2+4z^3)t^4+\cdots, \\
    \alpha &= z+2z(1-z)t+z(1-z)(8-z)t^2+32z(1-z)t^3+\cdots. \\
  \end{split}
\end{equation}
(It can again be shown that $y$ and $\alpha$ are series in $t$ whose
coefficients are polynomials in $z$ with integer coefficients, and it
seems that all coefficients of $y$ and $\alpha y$ are positive.)

By the mobile interpretation of $R_i$ explained in
Section~\ref{sec:2pbip}, we have
\begin{equation}
  R_i = 1 + t\, U_i T_{i+1}
\end{equation}
which extends \eqref{eq:2pbiprti} for $z \neq 1$. We deduce that
\begin{equation}
  R_i = R \frac{(1-\alpha y^{i+1})(1-\alpha y^{i+4})}{
    (1-\alpha y^{i+2})(1-\alpha y^{i+3})}, \qquad
  R = 1+t\, UT =\frac{(1-\alpha y^2)(1-\alpha y^3)}{(1-\alpha y)(1-\alpha y^4)}
\end{equation}
which matches the form \eqref{eq:2pbipcontfracform} expected for the
two-point function of a family of bipartite maps, according to the
continued fraction approach (our weighting scheme is indeed equivalent
to attaching a weight $z t^k$ to each face of degree $2k$).

Finally, by a straightforward generalization of the arguments given in Section~\ref{sec:genhyper},  we
may compute the generating function $\mathcal{R}_i=\mathcal{R}_i(t,z)$ of vertex-pointed general
hypermaps with a marked edge of type $(j-1,j)$ for $j\leq i$, enumerated now with both a weight $t$
per edge of the hypermap and a weight $z$ per dark face. It simply reads
\begin{equation}
\mathcal{R}_i= 1+ t^2 U_iU_{i+1}T_{i+2}\ ,
\end{equation} 
$i\geq 1$ (with as before a conventional term $1$).  Using
the explicit form of $T_i$ and $U_i$ above, we arrive at the factorized form
\begin{equation}
\begin{split}
&\mathcal{R}_i=\mathcal{R} \frac{(1-\alpha y^{i+2})(1-\alpha y^{i+4})}{(1-\alpha y^{i+3})^2}\\
&\text{where}\ \ \mathcal{R}=1+ t^2\,U^2\, T .
\end{split}
\end{equation}

\section{Conclusion}
\label{sec:conclusion}
In this paper, we made extensive use of bijections to obtain the two-point function for a number of families
of maps and hypermaps with unbounded faces degrees.
We now list a few remarks and discuss possible extensions of our work, as well as open questions.

First, we have found that the two-point function $R_i$ of general maps (resp.\ bipartite maps) counted by their number of edges 
is closely related to that of quadrangulations (resp.\ $3$-regular constellations, i.e. Eulerian triangulations), counted
by their number of faces (resp.\ dark faces). Indeed, we have $R_i=1+t T_{i}T_{i+1}$ where 
$T_{i}$, given by \eqref{eq:2pgenTi} (resp.\ \eqref{eq:biptisol}), and defined as a mobile generating function, can also be viewed as the two-point function for quadrangulations
(resp.\ Eulerian triangulations), see \cite{ONEWALL}. As a consequence, it is easily checked that these quantities are the
same in the scaling limit, namely the continuum two-point function $\varphi(r)$ of general maps (resp.\ bipartite maps)
\begin{equation}
\varphi(r)= \lim_{n\to \infty} \frac{[t^n] R_{\lfloor r n^{1/4} \rfloor}}{[t^n]R}
\end{equation}
coincides with that of quadrangulations (resp. Eulerian triangulations), which reads explicitly \cite{GEOD}
\begin{equation}
\begin{split}
\varphi(r)& =
\lim_{n\to \infty} \frac{[t^n] T_{\lfloor r n^{1/4} \rfloor}}{[t^n]T}\\
&= \frac{4}{\sqrt{\pi}}\int_0^\infty d\xi\, \xi^2\, e^{-\xi^2} \left( 1-6 \frac{1-\cosh(a\, r  \sqrt{\xi})
\cos(a\, r  \sqrt{\xi}) }{(\cosh(a\, r  \sqrt{\xi})-\cos(a\, r  \sqrt{\xi}))^2} \right)\\
\end{split}
\end{equation}
with $a=\sqrt{3}$ (resp.\ $a=\sqrt{2}$).
Note that $\varphi(r)$ with $a=\sqrt{2}$ is also the continuum
two-point function of general hypermaps $\varphi(r)= \lim_{n\to \infty} \frac{[t^n] \mathcal{R}_{\lfloor r n^{1/4} \rfloor}}{[t^n]\mathcal{R}}$.
The equality of the two-point functions of general maps and quadrangulations in the scaling limit suggests
the stronger property that these metric spaces are asymptotically the ``same'' (note that in both cases,
the two-point function corresponds to the actual, symmetric, graph distance). Precisely, we conjecture that the
Ambj\o rn-Budd bijection defines a coupling between uniform planar maps with $n$ edges and uniform planar
quadrangulations with $n$ faces such that their Gromov-Hausdorff distance is asymptotically almost surely negligible
with respect to $n^{1/4}$. Since the latter are known to converge to the Brownian map \cite{Miermont2013,LeGall2013},
the convergence of the former would follow.

Second, we have shown that the two-point functions of general and bipartite maps match the form expected
from the continued fraction approach. It would be interesting to have a full independent rederivation of our
results via this approach. This requires the computation of Hankel determinants of growing size: in the
context of maps with bounded face degrees \cite{BG12}, this was made through an identification of
these determinants as symplectic Schur functions. Here no such reduction is available and this is therefore
a challenging question.  

Third, we have exhibited in Section~\ref{sec:2p2par} explicit expressions for the two-point functions 
of general maps and bipartite maps counted by both their number of edges and their number of faces.
Our derivation makes use only of the case $p=2$ of Theorem \ref{theo:hyp_mobiles} and Proposition \ref{prop:const1}.
It is natural to expect an extension of these formulas for general $p$, i.e.\ for $p$-hypermaps and
$p$-constellations counted by both their number of dark faces and their number of light faces. In other words,
we expect the system of equations satisfied by the corresponding 2-parameter generating functions to be
still integrable. We have not however been able to guess its actual solution.
 
Returning to the case $p=2$, the extra control on the number of faces allows to interpolate naturally
between trees (maps with a single face) and arbitrary maps (with an unconstrained number of faces).
For a fixed large number $n$ of edges, the typical distance falls down from the known order $n^{1/2}$
for trees to $n^{1/4}$ for maps. 
We may then wonder about the possibility of obtaining new, non-generic, scaling limits for maps with ``few'' faces,
where the distances would be of order $n^\beta$ with $1/4<\beta<1/2$. We plan to investigate the question
in the future.

\section*{Acknowledgements} The work of
JB and \'EF was partly supported by the ANR projects ``Cartaplus''
12-JS02-001-01.

\bibliographystyle{hplain}
\bibliography{gen2p}

\end{document}